\documentclass[11pt,letterpaper]{amsart}

\usepackage[foot]{amsaddr}  

\usepackage{amsmath,amsfonts,amssymb}
\usepackage{amsthm,upref}
\usepackage{cancel}
\usepackage{geometry}
\usepackage{graphicx}
\usepackage{subcaption,tabularx,float}
\usepackage{hyperref}
\usepackage[capitalise]{cleveref}
\usepackage{enumerate}
\usepackage{xcolor}

\newcommand{\set}[1]{\left\lbrace #1 \right\rbrace}

\newcommand{\R}{\mathbb{R}}
\newcommand{\N}{\mathbb{N}}
\newcommand{\Q}{\mathbb{Q}}
\newcommand{\bS}{\mathbb{S}}
\newcommand{\T}{\mathbb{T}}
\newcommand{\Z}{\mathbb{Z}}
\newcommand{\B}{\mathcal{B}}
\newcommand{\D}{\mathcal{D}}
\newcommand{\G}{\mathcal{G}}

\renewcommand{\L}{\mathcal{L}}

\newcommand{\cO}{\mathcal{O}}
\newcommand{\cV}{\mathcal{V}}

\newcommand{\eps}{\varepsilon}
\newcommand{\vphi}{\varphi}

\makeatletter
\def\ip{\@ifnextchar\bgroup{\@with}{\@without}}
\def\@with#1#2{\langle #1,#2 \rangle}
\def\@without{\langle \cdot,\cdot \rangle}
\makeatother

\newcommand{\WB}{W\!B}

\DeclareMathOperator{\dig}{dig}

\DeclareMathOperator{\absdig}{absdig}
\DeclareMathOperator{\reldig}{reldig}

\theoremstyle{plain}
\newtheorem{thm}{Theorem} 
\newtheorem{proposition}[thm]{Proposition}

\theoremstyle{definition}
\newtheorem{definition}{Definition}

\theoremstyle{remark}
\newtheorem{remark}{Remark}

\newcommand{\Eq}[1]{\eqref{eq:#1}}
\newcommand{\Th}[1]{Th.~\ref{thm:#1}}

\newcommand{\Sec}[1]{\S \ref{sec:#1}}
\newcommand{\Fig}[1]{Fig.~\ref{fig:#1}}

\newcommand{\InsertFig}[4]
{\begin{figure}[h!t]
       \centerline{
         \includegraphics[width=#4\linewidth]{./#1}
       }
       \caption{{\footnotesize  #2}
       \label{fig:#3}}
\end{figure}}

\newcommand{\InsertFigTwo}[5] {
\begin{figure*}[h!t]
       \centerline{
         \includegraphics[width=#5\textwidth]{./#1}
         \hskip 0.2in
         \includegraphics[width=#5\textwidth]{./#2}
       }
       \caption{{\footnotesize  #3}
       \label{fig:#4}}
\end{figure*}}

\newcommand{\InsertFigSix}[9] {
\begin{figure}[h!t]
       \centerline{
\renewcommand{\arraystretch}{0.01}
         \begin{tabular}{ccc}
         \includegraphics[width=#9]{./#1}&  \includegraphics[width=#9]{./#2}  &
             \includegraphics[width=#9]{./#3} \\  
         \includegraphics[width=#9]{./#4} & \includegraphics[width=#9]{./#5}  & 
             \includegraphics[width=#9]{./#6}
        \end{tabular}
       }
       \caption{{\footnotesize #7}
       \label{fig:#8}}
\end{figure}}

\newcommand{\beq}[1]{\begin{equation}\label{eq:#1}}
\newcommand{\eeq}{\end{equation}}

\newenvironment{se}[1]{\equation\label{eq:#1}\aligned}{\endaligned\endequation}
\newcommand{\bsplit}[1]{\begin{se}{#1}}
\newcommand{\esplit}{\end{se}}

\begin{document}
\title[Weighted Birkhoff Averaging]{Distinguishing between Regular and Chaotic orbits of Flows by the Weighted Birkhoff Average}

\author{Nathan Duignan}
\address{Department of Mathematics and Statistics, University of Sydney, NSW, 2006, Australia}
\email{Nathan.Duignan@sydney.edu.au}
\author{James~D. Meiss} 
\address{Department of Applied Mathematics, University of Colorado, Boulder, CO, 80309-0526 USA}


\begin{abstract}
This paper investigates the utility of the weighted Birkhoff average (WBA) for distinguishing between regular and chaotic orbits of flows, extending previous results that applied the WBA to maps.
It is shown that the WBA can be super-convergent for flows when the dynamics and phase space function are
smooth, and the dynamics is conjugate to a rigid rotation with Diophantine rotation vector.
The dependence of the accuracy of the average on orbit length and width of the weight function width are investigated.
In practice, the average achieves machine precision of the rotation frequency of quasiperiodic orbits 
for an integration time of $\cO(10^3)$ periods.
The contrasting, relatively slow convergence for chaotic trajectories allows an efficient discrimination criterion. 
Three example systems are studied: a two-wave Hamiltonian system, a quasiperiodically forced,
dissipative system that has a strange attractor with no positive Lyapunov exponents,
and a model for magnetic field line flow.
\end{abstract}
\date{\today}

\maketitle

\tableofcontents 


\section{Introduction}\label{sec:Intro}

Integrability is associated with quasiperiodic dynamics and chaos with sensitive dependence on initial conditions. This contradistinction is especially relevant for smooth Hamiltonian systems: when such a flow is \textit{integrable} the orbits are confined to tori on which the dynamics is conjugate to a rigid rotation. When a Hamiltonian system is smoothly perturbed away from integrability, some of these tori persist---according to KAM theory---and some are replaced by isolated periodic orbits, islands, or chaotic regions \cite{Arnold78}. Typically as a perturbation grows the proportion of chaotic orbits increases and more of the tori are destroyed. 

Motivated by broad applications to dynamical systems, including fluid flow, the $n$-body problem, and toroidal magnetic confinement, there has been a concentrated effort to distinguish between chaotic regions of phase space and those with regular dynamics. Invariant tori in Hamiltonian systems can be computed as limits of periodic orbits \cite{Greene79, MacKay83} or by the iterative, parameterization method \cite{Haro06}. In these methods, one fixes a frequency vector and attempts to find an invariant set on which the dynamics has this frequency. 

More generally---even when the system is not Hamiltonian---one may try to detect when a given orbit is chaotic. By definition, a dynamical system is chaotic on a compact invariant set when it is transitive and exhibits ``sensitive dependence on initial conditions'' \cite{Auslander80, Meiss17a}. Often such dynamics are (nonuniformly) hyperbolic, meaning that the maximal Lyapunov exponent is positive \cite{Robinson99}. For a flow
\beq{flow}
	\vphi: \R \times X \to X, \quad x_t = \vphi_t(x_0),
\eeq
on a phase space $X$, the exponent is
\beq{Lyapunov}
	\lambda(x_0,v_0) = \limsup_{T \to \infty} \frac{1}{T} \ln \|D\vphi_t(x_0) v_0 \|
\eeq
for initial condition $x_0$ and initial deviation vector $v_0$. A positive Lyapunov exponent implies that the
length of the infinitesimal deviation grows exponentially in $T$, at least asymptotically.
The most common approach for distinguishing chaos from regularity is to numerically compute \Eq{Lyapunov};
however, accurate computation of $\lambda$ is difficult because convergence is typically
as slow as $\tfrac{\ln(T)}{T}$ \cite{Cincotta16}.
Computation of \Eq{Lyapunov}  is also expenstive because it is necessary to integrate both the trajectory and the linearized dynamics to obtain the Jacobian $D\vphi_t$.

There are a number of techniques that have been used to improve the efficiency of estimates
for exponential divergence. These include methods based on \Eq{Lyapunov} such as the 
Fast Lyapunov Indicator (FLI) \cite{Froeschle97a, Lega16}, which uses a large value of
\beq{FLI}
	\operatorname{FLI}(x_0,v_0) = \sup_{t<T} \ln \|D\vphi_t(x_0) v_0 \|
\eeq
as an indicator for chaos. A related idea, the Mean Exponential Growth factor of Nearby Orbits (MEGNO)
\cite{Giordano04, Cincotta16}, uses---instead of the supremum in \Eq{FLI}---the average
of this log-length along an orbit. Further techniques include computing Greene's residue \cite{Greene79},
Slater's method \cite{abudSlaterCriterionBreakup2015}, the 0-1 test \cite{gottwald01TestChaos2016},
SALI and GALI \cite{skokosSmallerSALIGeneralized2016}, expansion entropy \cite{huntDefiningChaos2015},
and converse KAM theory \cite{MacKay85,Duignan20}. 



In this paper, we explore an alternative technique to distinguish between chaotic and regular orbits of
a flow based on the Weighted Birkhoff Average (WBA) \cite{Das16a,Das16b,Das17}. This method permits one
to accurately and efficiently compute the average of a function $h: X \to \R$, when the orbit is regular.
In particular, $h$ can be chosen to give the rotation vector for regular orbit on an invariant torus,
so that it can also provide a distinction between resonant and quasiperiodic dynamics.
The technique is analogous to ``frequency analysis'' \cite{Laskar92, Bartolini96}, which uses a windowed
Fourier transform to compute rotation numbers. For the WBA, the choice of a smooth window or weight function
allows for more rapid convergence. 

Indeed, as we recall in \Sec{Birkhoff}, the method proposed in \cite{Das16a} uses a $C^\infty$ weight function, which has been shown by \cite{Das18b} to lead to super-polynomial convergence of the average for maps. 
We will generalize these results for flows in \Sec{SuperConvergence}. In \Th{superconvergence}, we establish the super-polynomial convergence of the WBA to the space average provided that the flow is quasi-periodic on an $n$-torus with Diophantine rotation vector. Theorem~\ref{thm:general_superconvergence} extends this, using the results of \cite{kachurovskiiMaximumPointwiseRate2021}, to give a weaker criteria for super-polynomial convergence.

In \Sec{WBA_Test} we use the distinction between convergence rates for regular and chaotic orbits to give a criterion for detecting chaos. The method is analogous to that used in \cite{Sander20,Meiss21} for maps. Then, in \Sec{applications}, we apply this test to three examples.

The first application, in \Sec{twoWaveModel}, is to the two-wave model, perhaps the simplest nonintegrable, $1\tfrac12$ degree-of-freedom Hamiltonian system. The model was also studied in \cite{MacKay89a,Duignan20} using the converse KAM theory to detect chaos, and consequently gives a contrast between the two methods. We also investigate the dependence of the accuracy of the WBA on the choice of orbit length and weight function width. 

In \Sec{QPendulum} we investigate the properties of the WBA for a quasiperiodically forced pendulum that can have geometrically strange attractors with no positive Lyapunov exponents \cite{Romeiras87}. Since the general definition of chaos requires only the topological form of sensitive dependence on initial conditions and not exponential divergence \cite{huntDefiningChaos2015}, orbits of this quasiperiodic system with zero exponents may still be chaotic \cite{Glendinning06}. As we will show, even though methods based on Lyapunov exponents would fail,
the WBA can still provide an efficient indicator of chaos, 

The final application, in \Sec{FieldLines}, is to magnetic field line flow. 
Integrable magnetic field-line configurations are desirable in the design of plasma confinement devices. 
For example, the tokamak is designed to have a set of nested, axisymmetric tori
that are tangent to the magnetic field and correspond to
iso-pressure surfaces \cite{Hazeltine03}. 
However, integrability can be destroyed by instabilities or non-axisymmetric perturbations that can
give rise to magnetic islands and chaotic regions \cite{Helander14}. Using a model introduced by
\cite{paulHeatConductionIrregular2022}, we will show that the weighted Birkhoff
average can rapidly and accurately measure the extent of these regions.
Moreover, in the study of plasma stability it is important to know the rotation number, 
or rotational transform, on each magnetic surface. 
We will demonstrate that the weighted Birkhoff average efficiently computes the rotational transform.

\section{Weighted Birkhoff Averaging for Flows}\label{sec:Birkhoff}

\subsection{Birkhoff Average}\label{sec:BA}
Suppose that $(X,\B,\mu)$ is a probability space \cite{Walters82}, 
and that the measure $\mu$ is invariant under a flow \Eq{flow}.
The flow is \textit{ergodic} with respect to $\mu$ if, whenever $S \in \B$ is invariant,
i.e., $\vphi_t(S) = S,\, \forall t\in \R$, then $\mu(S) = 0$ or $1$.
Thus invariant sets are either of zero measure, such as periodic orbits, or of full measure, such as chaotic orbits.
Birkhoff's ergodic theorem \cite{Birkhoff31b}, states that if a flow is ergodic then for $\mu$-almost every $x \in X$, the time average of a function $h$,
\beq{BirkhoffAverage}
	B_T(h)(x) = \frac{1}{T}\int_{0}^{T} h(\vphi_t(x)) dt,
\eeq
converges to its space average
\beq{spaceAverage}
	\langle h \rangle \equiv \int_X h d\mu
\eeq
as $T\to\infty$.
\begin{thm}[Birkhoff Ergodic Theorem \cite{Birkhoff31a,Birkhoff31b}]\label{thm:BET}
Suppose $(X,\B,\mu)$ is a probability space, $\vphi_t: X \to X$ is a measure-preserving, ergodic flow, and $h \in L^1(X,\mu)$. Then the time average exists and
\[ 
	 \lim_{T\to \infty} B_T(h)(x) = \langle h \rangle
\] 
for $\mu$-almost every $x \in X$.
\end{thm}

The Birkhoff ergodic theorem for a map $F:X \to X$ can be obtained upon replacing the integration in \Eq{BirkhoffAverage} by a sum,
\[
	B_N(h)(x) = \frac{1}{N}\sum_{j=0}^{N-1} h( F^j(x)).
\]
Indeed, in the literature, \Th{BET} is almost exclusively stated and proven for maps \cite{ Billingsley65, Cornfeld82, Breiman92}. However, as pointed out in \cite{bergelsonDiscreteContinuoustimeErgodic2012}, there is a neat trick to obtain the continuous case from the map case. 

This is based on the result that if $\vphi_t(x)$ is an ergodic flow then for each $\tau \in \R$, except for a countable subset, the map, $\vphi_\tau$, is also ergodic \cite{pughErgodicElementsErgodic1971}.
For such a value of $\tau$, define $\tilde{h}(x) = \int_{0}^{\tau} h(\vphi_t(x)) dt$. Then
\[ 
	\int_0^T h(\vphi_t(x)) dt = \sum_{j=0}^{\lfloor T/\tau \rfloor} \tilde{h}(\vphi_\tau^j(x))+ \int_{\lfloor T/\tau \rfloor}^T h(\vphi_t(x)) dt,
\]
since $\vphi_\tau^0(x) = x$.
Since $h$ is assumed to be $L^1(X,\mu)$, then $\lim_{T\to\infty}\tfrac{1}{T} \int_{\lfloor T/\tau \rfloor}^T h(\vphi_t(x)) dt = 0$. By the Birkhoff ergodic theorem for maps, it then follows that
\begin{align*}
	\lim_{T\to\infty} \frac{1}{T}\int_0^T h(\vphi_t(x)) dt 
	  &= \frac{1}{\tau}\lim_{T \to \infty} \frac{1}{T/\tau} \sum_{j=1}^{\lfloor T/\tau \rfloor} \tilde{h}(\vphi_\tau^j(x)) \\
	  &= \frac{1}{\tau}\int_{X} \tilde{h} d\mu 
	   = \frac{1}{\tau} \int_X \int_0^\tau h(\vphi_t(x)) d\mu \\
	  &= \frac{1}{\tau} \int_0^\tau \int_X h(x) d\mu \\
	  &= \int_X h d\mu .
\end{align*}
Here the penultimate equality is a consequence of Fubini's theorem and the fact that $\mu$ is invariant under $\vphi_t$.

Even though the convergence in \Th{BET} is guaranteed, it can be arbitrarily slow depending on the choice of $h$ \cite{Krengel78}. Furthermore for almost all $h\in L^1(X,\mu)$, it has been demonstrated that for maps the convergence is at most $\cO(1/N)$ \cite{Kachurovskii96}, and for flows on a Lebesgue space is at most $\cO(1/T)$ for almost all $x\in X$ \cite{kachurovskiiMaximumPointwiseRate2021}. The only exceptions are when $h$ is almost everywhere constant. 

\subsection{Weighted Birkhoff Average}\label{sec:WBA}

A weighted Birkhoff average is analogous to \Eq{BirkhoffAverage}, with the addition of a weight function $g: [0,1] \to [0,\infty)$, that is normalized:
\beq{Normalization}
	||g||_1 \equiv \int_0^1 g(s) ds = 1.
\eeq
For any $g\in\G$, the weighted Birkhoff average is defined by
\beq{WeightedBirkhoff}
	\WB_T(h)(x) = \frac{1}{T} \int_0^T g(\tfrac{t}{T}) h(\vphi_t(x)) dt.
\eeq
In particular, choosing $g(s) =1 $ gives the Birkhoff average \Eq{BirkhoffAverage}.

By a judicious choice of the weight $g$, the convergence of $\WB_T(h)$ to the space average can be accelerated for certain flows. In particular, we will consider the space of bump functions whose first $m$ derivatives vanish on the boundary,
\beq{Gboundary}
	\G_m = \left\{\left. g\in C^m([0,1],\R^+)\, \right| \, \|g\|_1 = 1, \,g^{(i)}(0) = g^{(i)}(1) = 0,\, i = 0, 1, \ldots, m-1 \right\}.
\eeq
For example, the smooth bump function
\beq{Bump}
	g(s) = \begin{cases}
				C e^{-[s(1-s)]^{-1} } & s\in(0,1) \\
				0 & s=0,1 \\
			\end{cases},
\eeq
is in $\G_\infty$ and was adopted by \cite{Das16b,Das17,Das19} in their studies of maps.
Here, we set the normalization constant $C \approx 142.2503758$ to satisfy \Eq{Normalization}.

Whenever $g \in \G_1$ it can be shown, using a result of \cite{silvermanNotionSummabilityLimit1916},
that the weighted Birkhoff average $\WB_T(h)(x)$ converges to the space average $\langle h \rangle$ 
for any $h\in L^1(X,\mu)$. Of course this applies to the case \Eq{Bump} as well.

\begin{proposition} Under the hypotheses of \Th{BET}, then whenever $g\in \G_1$
\[ 
	\lim_{T\to \infty} \WB_T(h) = \langle h \rangle,
\]
for $\mu$-almost every $x \in X$.
\end{proposition}
\begin{proof}
	The proof relies on the summation criteria due to Silverman \cite{silvermanNotionSummabilityLimit1916}. Specifically, suppose that $k(T,t)$ is defined for $T\in\R^+$ and $0 \leq t \leq T$, is integrable for fixed $T$, and satisfies the three criteria:
\begin{subequations}\label{eq:Silverman}
	\begin{align}
		&\lim_{T\to\infty} \int_0^T k(T,t) dt = 1 ; \label{eq:Silverman1}\\
		&\lim_{T\to\infty} k(T,t) = 0 \text{ uniformly in $t$ on $[0,q]$ for each $q\in\R^+$}; \label{eq:Silverman2}\\
		&\int_0^T|k(T,t)| dt < A \text{ for some $A>0$ whenever $T\in \R^+$}. \label{eq:Silverman3}
	\end{align}
\end{subequations}
Then Silverman \cite[Thm.~1]{silvermanNotionSummabilityLimit1916} shows that whenever $u:\R^+ \to \R$ satisfies $\displaystyle\lim_{T\to\infty} u(T) = \bar{u} <\infty$
\[ 
	\lim_{T\to\infty} \int_0^T k(T,t) u(t) dt = \bar{u}.
\]

We will take $u(T) = B_T(h)(x)$ given by \Eq{BirkhoffAverage}. Using integration by parts on \Eq{WeightedBirkhoff},
and the assumptions that $g \in C^1([0,1],\R^+)$ and $g(0) = 0 = g(1)$, then gives
\bsplit{WBtoSilver}
	\WB_T(h)(x) &= \frac{1}{T} \int_0^T g(\tfrac{t}{T}) h(\vphi_t(x))\, dt \\
				&= \frac{1}{T}\left[ g(\tfrac{t}{T}) \int_0^t h(\vphi_s(x))\, ds \right]_0^T 
				 - \frac{1}{T^2} \int_0^T g^\prime(\tfrac{t}{T}) \int_0^t h(\vphi_s(x))\, ds \, dt \\
				&= \int_0^T k(T,t) B_t(h)(x)\, dt.
\esplit
where we set
\beq{SilvermanK}
	k(T,t) = -\frac{t}{T^2} g^\prime(\tfrac{t}{T}) . 
\eeq
We will check Silverman's criteria for \Eq{SilvermanK}. Again using integration by parts we have
\[ 
	\int_0^T \frac{t}{T^2} g^\prime(\tfrac{t}{T}) \, dt = \left[\frac{t}{T} g(\tfrac{t}{T}) \right]_0^T 
		+ \frac{1}{T} \int_0^T g(\tfrac{t}{T}) dt = 1,
\]
since $\|g\|_1 = 1$. It follows that $k$ is integrable for each $T\in\R^+$ and that \Eq{Silverman1} holds.
	
Now, for any $q >0$, $0 \le t \le q <T$ we have that 
\[ 
	|k(T,t)| = \frac{t}{T^2} |g^\prime(\tfrac{t}{T})| \leq \frac{q}{T^2}||g^\prime||_\infty , 
\]
where $||g^\prime||_\infty < \infty$ because $g$ is continuously differentiable
on a compact interval. It follows that $\lim_{T\to\infty} k(T,t) = 0$ and thus \cref{eq:Silverman2} holds.

Finally, for all $T \in \R^+$, \Eq{SilvermanK} gives
	\[ \int_0^T |k(T,t)| dt \leq ||g^\prime||_\infty T^{-2} \int_0^T t\, dt = \tfrac12 ||g^\prime||_\infty. \]
	Thus \Eq{Silverman3} holds with $A = \tfrac12 ||g^\prime ||_\infty$.

Hence, the choice \Eq{SilvermanK} satisfies the criteria \Eq{Silverman}. 
Since by Birkhoff's ergodic theorem, $B_T(h)(x) \to \langle h \rangle$ for $\mu$-almost all $x$,
then \cite[Thm.~1]{silvermanNotionSummabilityLimit1916} applied to \Eq{WBtoSilver}, 
gives $\WB_T(h)(x) \to \langle h \rangle$.
\end{proof}

\section{Super-Convergence for Flows} \label{sec:SuperConvergence}

It was shown in \cite{Das17, Das18b} that if a smooth map $F$ has a quasi-periodic orbit 
$\{F^t(x)\}$ with Diophantine rotation vector and $h$ and $g$ are $C^\infty$, 
then the weighted average \Eq{WeightedBirkhoff} is \emph{super-polynomial convergent}: $\WB_T(h)$ converges 
to $\langle h \rangle$ faster than any power of $T$. 
Super-convergence is especially useful for the case of Hamiltonian flows or symplectic maps, where regular orbits lie on invariant tori, and KAM theory implies the structural stability of those with Diophantine rotation vectors.
 
In this section we will extend the map result to the case of flows. In addition, following \cite{kachurovskiiMaximumPointwiseRate2021}, we will show that super-polynomial convergence also holds under weaker hypotheses on the ergodic flow and $h$.

Note that if the weight function $g$ has only finitely many vanishing derivatives at the endpoints, \Eq{Gboundary},
then the convergence rate is $T^{-m}$. For example the weight $g(t) \propto \sin^2(\pi t)$ for $t \in (0,1)$, 
is first order smooth, but not second order, since $g^{(2)}(0^+) \neq g^{(2)}(0^-)$, and this implies that the convergence is $\cO(T^{-2})$.

By contrast, the weighted average appears to converge only as $T^{-1/2}$ for chaotic orbits \cite{Levnajic10, Sander20}. Thus, as we will discuss in \Sec{WBA_Test}, the convergence rate of the weighted average can provide a useful distinction between regular and chaotic orbits \cite{Sander20,Meiss21}.

\begin{definition}[Super-Convergent]
A function $ f:[0,\infty) \to \R$ with $\lim_{T\to\infty} f(T) = f^* < \infty$ is \emph{super-polynomial} convergent if, for each $m \in \N$, there is a constant $c_m > 0$ such that 
\[ 
	|f(T) - f^*| \leq c_m T^{-m}, 
\]
for all $T > 0$.
\end{definition}
In particular the weighted Birkhoff average super-converges to the space average for flows that are conjugate to a rigid rotation with a sufficiently irrational vector, e.g., one that satisfies a Diophantine property.

\begin{definition}[Diophantine \cite{Lochak92}]\label{def:Diophantine}
A vector $\omega \in \R^d$ is \emph{Diophantine} if there is a $c>0$ and $\tau \ge d-1$ such that 
\beq{Diophantine}
	\omega \in \D_{c,\tau} := \set{\omega\in\R^d\, | \, 
		|k\cdot\omega| > c \|k\|^{-\tau},\, \forall k \in \Z^d_0}
\eeq 
where $\Z^d_0 = \Z^d\setminus 0$. 
\end{definition}

Thus, for example, the vector $\omega = (\phi,1) \in \R^2$, where $\phi = \tfrac12(\sqrt{5}-1)$ is the (inverse of)
the golden mean, is in $\D_{1/\sqrt{5},1}$. More generally, there are bounds for the Diophantine constants for integral bases of an algebraic field of degree $d$ \cite{Cusick74, Cassels97}.\footnote
{Note that for the discrete time case, a $k$-dimensional rotation vector $\alpha$ is Diophantine
if the vector $\omega = (\alpha,1) \in \R^{k+1}$ is Diophantine in the sense of \Eq{Diophantine}.}

With these definitions, we can restate the result of \cite{Das18b} for quasiperiodic flows.
\begin{thm}\label{thm:superconvergence}
Let $M \simeq \T^d$ be a smooth manifold and $\vphi_t:M\to M$ be a smooth, quasi-periodic flow with invariant probability measure $\mu$. Assume $\vphi_t$ is $C^\infty$ conjugate to a rigid rotation with a Diophantine rotation vector $\omega \in \D_{c,\tau}$. Suppose that $h \in C^\infty(M, \R)$ and $g \in \G_\infty$. Then for each $x \in M$, the weighted Birkhoff average \Eq{WeightedBirkhoff} is super-polynomial convergent. Moreover, the convergence is uniform in $x$.

More generally, if $g \in \G_m $ then the convergence of \Eq{WeightedBirkhoff} is as $T^{-m}$ provided that $h \in C^l(M,\R)$ for some $l>d+m\tau$.
\end{thm}

\begin{proof}
	The proof follows the arguments of \cite{Das18b} with some minor alterations for the flow case. By assumption $\vphi$ is diffeomorphic to the flow $\vphi_t^\omega(\theta) = \theta + t\omega$ on $\T^d$. Hence, it can be assumed that we have taken coordinates $\theta \in \T^d$ on $M$ so that $\vphi_t$ is simply $\vphi_t^\omega$ and the invariant measure $\mu$ becomes the constant measure $d\theta$, which is preserved by $\vphi_t^\omega$.
Then
\[ 
	 h(\vphi_t(\theta)) = h(\theta + t\omega).
\]

The weighted Birkhoff average \Eq{WeightedBirkhoff} then becomes
\[
	\WB_T(h)(\theta) = \frac{1}{T}\int_0^T g(\tfrac{t}{T}) h(\theta + t\omega) dt.
\]
Since $h$ is $L^2$, it has a Fourier series
\[ 
	h = \sum_{k\in\Z^d} a_k e^{2\pi i k \cdot \theta}.
\]
that is almost everywhere convergent \cite{Grafakos14}.
Note that $a_0 = \int_{M} h d\theta $ and that $a_0 = \WB_T(a_0)(\theta_0)$ for any $\theta_0$ since $a_0$ is constant, for any $T$ and $\theta_0$.
It follows that 
\begin{align*}
	\left|\WB_T(h)(\theta) - \int_M h d\theta \right| &= \left| \sum_{k\in \Z^d_0} a_k \WB_T\left(e^{2\pi i k\cdot \theta}\right) \right| \\
	&= \frac{1}{T} \left| \sum_{k\in\Z^d_0} a_k e^{2\pi i k \cdot \theta} \int_0^T g(\tfrac{t}{T}) e^{2\pi i t k\cdot \omega} dt \right| \\
	& \leq \frac{1}{T} \sum_{k\in\Z^d_0} |a_k| \left|\int_0^T g(\tfrac{t}{T}) e^{2\pi i t k\cdot \omega } dt \right|
\end{align*}
Set $s = t/T$ so that $T ds = dt$ and define $\Omega_k = 2\pi T k \cdot \omega$. Then
\begin{equation*}
	\left|\WB_T(h)(\theta) - \int_M h d\theta \right| \leq \sum_{k\in \Z^d_0} |a_k| \left| \int_0^1 g(s) e^{i \Omega_k s} ds\right|.
\end{equation*}
	
Integrating by parts $m \le l$ times and noting that the boundary terms vanish by \Eq{Gboundary},
it follows that
\begin{align*}
	\left| \int_0^1 g(s)e^{i\Omega_k s} ds \right| &= |\Omega_k|^{-m} \left| \int_0^1 g^{(m)}(s) e^{i\Omega_k s} ds \right| \\
	& \leq |\Omega_k|^{-m} ||g^{(m)}||_1,
\end{align*}
where $|| \,||_1$ is the $L^1$ norm. 
	
Now, if $h \in C^{l}(\T^d,\R)$ then $|a_k| = \cO(\|k \|^{-l})$, so that there
is a constant $c_l>0$, independent of $k$, such that
\[ 
	|a_k| \leq \frac{c_l}{\|k\|^{l}},
\]
for each $k\in\Z^d_0$ \cite[Thm 3.3.9]{Grafakos14}. 
Thus,
\begin{align*}
	\left|\WB_T(h)(\theta) - \int_M h d\theta \right| &\leq c_l (2\pi T)^{-m}\sum_{k\in\Z^d_0} \|k\|^{-l} |k\cdot \omega|^{-m} ||g^{(m)}||_1 \\ 
	&= C^* T^{-m} \sum_{k\in\Z^d_0} \|k\|^{-l} |k\cdot \omega|^{-m},
\end{align*}
where we defined $C^* = c_l (2\pi)^{-m} ||g^{(m)}||_1$, which depends on $m$. The theorem will follow provided we can show that $ \sum_{k\in\Z^d_0} \|k\|^{-l} |k\cdot \omega|^{-m} $ is bounded.
Since $\omega \in \D_{c,\tau}$ \Eq{Diophantine} it follows that
\[
	\sum_{k\in\Z^d_0} \|k\|^{-l} |k\cdot \omega|^{-m} < c^{-m} \sum_{k\in\Z^d_0}\|k\|^{m\tau - l}.
\]
Finally, noting that whenever $l - m\tau > d$, the sum above converges, we have the desired bound. That is the weighted Birkhoff average converges as $T^{-m}$ provided that $h \in C^l$ with $l > d+m\tau$.
\end{proof}


\begin{remark}
The weight function, $g(t) \propto \sin^2(\pi t)$ on $[0,1]$, is in $\G_1$ but not 
$\G_2$ since then its second derivative is not continuous on the boundary. Nevertheless,
the integration-by-parts in the proof of the theorem can proceed up to $m = 2$, 
since the boundary terms vanish for $g$ and $g'$ and $g^{(2)} \in L^1$. 
Thus in this case when $h \in C^l$ and $l > d+2\tau$, 
the convergence is as $T^{-2}$. So for example, in \cite{Das17}, $d=2$, and the rotation vector 
$\omega = (\sqrt{2}-1,1)$ is Diophantine with $\tau = 1$. So the convergence is $m = 2$ whenever
$l >4$. They consider a case with $h \in C^\infty$ and numerically
observe a bit faster convergence, as $T^{-2.5}$.
\end{remark}

Following the work of \cite{kachurovskiiMaximumPointwiseRate2021}, super-polynomial convergence can be guaranteed under weaker hypothesis on the flow than those in \Th{superconvergence}, provided $h$ has a particular structure.
\begin{thm}\label{thm:general_superconvergence}
	Let $(M,\B,\mu)$ be a probability space and $\vphi_t:X\to X$be a smooth, ergodic flow with invariant probability measure $\mu$. Suppose that $h \in L^1(X,\mu) $ and there exists bounded funccton $H\in C^m(X,\R)$ such that 
\beq{cohomologous}
	h(x) = \langle h \rangle -\left.\frac{d^m}{dt^m} H(\vphi_t(x)) \right|_{t=0},
\eeq
for all $x\in \bar{M}$ where $\bar{M}$ is an invariant set of full measure.
Moreover, suppose the bump function $g \in \G_\infty$.
Then, there exists a constant $c_m$ such that, 
\[
	\left|\WB_T(h)(x) - \langle h \rangle \right| \leq \frac{c_m}{T^m}\quad \text{for all $x \in \bar{M}$.} 
\]
\end{thm}
\begin{proof}
	As $\bar{M}$ is invariant, the semi-group property of $\vphi_t$ guarantees, for all $x\in \bar{M},\, \tau\in \R$, 
	\[
		h( \vphi_\tau(x)) = \langle h \rangle -\left. \frac{d^m}{dt^m} H (\vphi_{t+\tau} (x)) \right|_{t=0} = \langle h \rangle -\left. \frac{d^m}{dt^m} H (\vphi_{t} (x)) \right|_{t=\tau}.
	\]
	Using this relation we can rewrite the weighted Birkhoff average as,
	\begin{align*}
		\left|\WB_T(h)(x)-\langle h\rangle\right| &= \left| \frac{1}{T} \int_0^T g\left(\frac{t}{T}\right) h(\vphi_t(x))\, dt - \langle h \rangle\right| \\ 
			&= \left| \int_0^1 g(s) h( \vphi_{Ts}(x))\, ds - \langle h \rangle \right| \\
			&= \left| \int_0^1 g(s) \left. \frac{d^m}{dt^m} H (\vphi_{t}) (x) \right|_{t=sT}\, ds \right|. \\
	\end{align*}
	Finally, integrating by parts $m$ times and noting that the boundary terms vanish by \Eq{Gboundary} shows that, for all $x\in \bar{M}$,
	\begin{align*}
		\left|\WB_T(h)(x) - \langle h \rangle \right| &= \frac{1}{T^m}\left| \int_0^1 g^{(m)}(s) H (\vphi_{sT} (x))\, ds \right| \\
			&\leq \frac{1}{T^{m}} ||g^{(m)}||_1 ||H||_\infty. 
	\end{align*}
	Since $||g^{(m)}||_1 < \infty$ and $||H||_\infty < \infty$ by assumption, we can take $c_m = ||g^{(m)}||_1 ||H||_\infty $ to give the result. 
\end{proof}

\begin{remark}
	The condition \Eq{cohomologous} is equivalent to a function being `cohomologous' to it's average \cite{kachurovskiiMaximumPointwiseRate2021}, with some further regularity assumed. It is difficult to ascertain whether, for a given function $h$, there exists a function $H$. This is due to the fact that finding $H$ in \Eq{cohomologous} requires knowledge of both the function $h$ and the flow $\vphi_t$. Numerical observations in \cref{sec:applications} imply that for a chaotic orbit super-polynomial convergence is not observed. This indicates it is impossible to find the desired $H$ for a chaotic orbit. 
\end{remark}

\begin{remark}
	\Th{superconvergence} can be obtained as a corollary of \Th{general_superconvergence}. As in the proof of \Th{superconvergence}, assume $h \in \L^2(\T^n)$ and assume coordinates $\theta\in\T^n$ is taken so that $\vphi_t(\theta) = \theta + t\omega$. Denoting the Fourier series for $h$ by $h = \sum_{k\in \Z^d} a_k e^{2\pi i k \cdot\theta}$, we can then set
\[ 
	H = \sum_{k\in\Z^d} \frac{a_k}{(2\pi i k\cdot \omega)^m} e^{2\pi i k\cdot \theta }. 
\]
Then, for each $m$, $h$ is of the form \Eq{cohomologous}. Note that as $m$ grows, $H$ is bounded provided $\omega$ is Diophantine.
\end{remark}
\section{Weighted Birkhoff Average as a Test for Chaos}\label{sec:WBA_Test}

The Weighted Birkhoff average can be used as a test for chaos by examining the convergence rate of $\WB_T(h)(x)$ to $\langle h \rangle$ as $T \to \infty$, for a given function $h$. As discussed in \Sec{SuperConvergence}, if the an orbit is quasiperiodic with Diophantine rotation vector, the weighted Birkhoff average is super-convergent. By contrast, it is observed that the convergence rate for chaotic orbits is much slower. 

Following \cite{Sander20,Meiss21} we estimate the convergence rate by comparing the values of \Eq{WeightedBirkhoff} for successive time intervals of a fixed length. That is, an estimate for the accuracy of \Eq{WeightedBirkhoff} for a time $T$ is found by first computing $\WB_T(h)(x_0)$, and then $\WB_T(h)(x_T)$, the average along a second time $T$ segment that begins at $x_T = \vphi_T(x_0)$. A comparison of these values gives an estimate of the `error' in the time $T$ average relative to the true average for the initial point $x_0$.

We consider two primary options to quantify this error. The first,
\beq{absdig}
	\absdig_T(h)(x_0) \equiv -\log_{10}|\WB_T(h)(x_0) - \WB_T(h)(x_T)|,
\eeq
 we call the \textit{absolute digit accuracy}.
This is the measure proposed in \cite{Sander20, Meiss21}, where it was denoted simply by $\dig_T$. Equation~\Eq{absdig} is the number of digits that two segments of the average have in common. It should be useful when the expected averages are of the same magnitude.
A second error quantification is the \textit{relative digit accuracy},
\beq{reldig}
	\reldig_T(h)(x_0) \equiv -\log_{10}\frac{\left| \WB_T(h)(x_0) - \WB_T(h)(x_T)\right|}{\tfrac12 (|\WB_T(h)(x_0)| + |\WB_T(h)(x_T)|)}.
\eeq
This measures the number of digits relative to an expected value, estimated as the mean absolute value of the two partial averages. This should be useful when the average of $h$ varies widely in magnitude as $x_0$ varies. Note, however, that it is not a good quantification if the average is expected to be zero,
in which case \Eq{absdig} is more appropriate. Finally, we define the \textit{maximum digit accuracy}
\beq{maxdig}
	 \dig_T \equiv \max\{\absdig_T,\reldig_T\} .
\eeq

By comparing the values of these measures for different initial conditions, we can differentiate between orbits for which the averages converge more rapidly than others and hence, classify which orbits are chaotic and which are regular.

\section{Applications} \label{sec:applications}

In this section we apply the method outlined above to three example flows: the two-wave model \cite{Escande81}, a quasi-periodically forced vector field with a strange attractor \cite{Romeiras87}, and a model for magnetic fields investigated in \cite{paulHeatConductionIrregular2022}.

To numerically integrate these examples, we use the algorithm \textsc{Vern9}, in the Julia package \textsc{DifferentialEquation.jl} \cite{rackauckas2017differentialequations}. Each integration is computed with multiple-precision arithmetic, choosing absolute and relative tolerances between $10^{-10}$ and $10^{-15}$, depending on the differential equation. The WBA in \Eq{WeightedBirkhoff} was calculated by numerically solving the differential equation 
\beq{WBA_DiffEq}
	\frac{d}{dt} W(t) = g\left(\tfrac{t}{T}\right) h(\vphi_{t}(x_0)),
\eeq
where $g$ is the bump function \Eq{Bump}. Integration of \Eq{WBA_DiffEq} is simply done by adding it to the set of differential equations defining the flow $\vphi_{t}$. Thus \Eq{WBA_DiffEq} is integrated with the same algorithm as the trajectory. Finally we compute the time $T$ weighted Birkhoff average as $ \frac{1}{T} W(T) $.
This allows us to adjust the accuracy of the integrator so that the computation of \Eq{WeightedBirkhoff} has an accuracy comparable to the trajectory itself; indeed, the \textsc{Vern9} algorithm will adjust its step-size to ensure this accuracy.


\subsection{Two-Wave Model}\label{sec:DoubleWave}\label{sec:twoWaveModel}

Our first example corresponds to the 1D motion of a charged particle in the electric field two longitudinal electrostatic waves \cite{Escande81}. This \textit{two-wave model} was used in \cite{MacKay89a, Duignan20} to illustrate the so-called \textit{converse KAM} method that detects the breakup of tori. Since the destruction of tori is a signal of the onset of chaos, we can use this model to compare the efficiency of the weighted Birkhoff and converse KAM methods as chaos detectors.


Following \cite{Escande81}, the two-wave system has the nonautonomous Hamiltonian 
\beq{TwoWaveHam}
	H(q,p,t) = \tfrac12 p^2 - \mu \cos(2\pi q) - \mu \cos(2\pi (q-t)),
\eeq
for the position $q$ and momentum $p$ of the particle.
Here, without loss of generality, we choose the mass of the particle to be one, the phase velocities of the two waves to be zero and one, respectively, and the wavenumber of the first wave to be one.
For simplicity, we follow \cite{MacKay89a,Duignan20} to assume that the wavenumber of the second wave is also one and that that the two waves have the same amplitude, $\mu$. Thus this simplified model has only one parameter.
By taking $q$ and $t \mod 1$, the extended phase space can be thought of as 
$(q,p,t) \in \T \times \R \times \T$. 

A Poincar\'{e} section at $t = 0 \mod 1$ is shown in
\Fig{twoWavePoincare} for $\mu = 0.03$. 
When $\mu \ll 1$, most orbits lie on rotational, invariant 2D tori
(i.e., tori that are homotopic to the set $p=0$).
For small positive $\mu$, there are two primary
elliptic periodic orbits that cross the Poincar\'e section near $(0,0)$ and $(0,1)$ and two
hyperbolic periodic orbits crossing the section near $(\tfrac12,0),(\tfrac12,1)$.
Each of these four orbits correspond to fixed points of the Poincar\'e map.
The section for the range $p \in [-0.2,0.5]$ shown in \Fig{twoWavePoincare} shows only orbits
trapped in the stationary wave, near $p=0$. The 2D tori encircling the primary elliptic orbits
are \textit{librational}; they correspond to particles trapped in one of the two electrostatic waves. 
Also shown in \Fig{twoWavePoincare} are other resonant islands; these correspond to orbits trapped
near elliptic periodic orbits with rational winding numbers on $\T^2$. The largest seen in the
figure is a pair of islands surrounding a period-two orbit near $p = 0.5$ on the section.

\begin{figure}[ht]
	\centering
	\includegraphics[width = 0.8\linewidth]{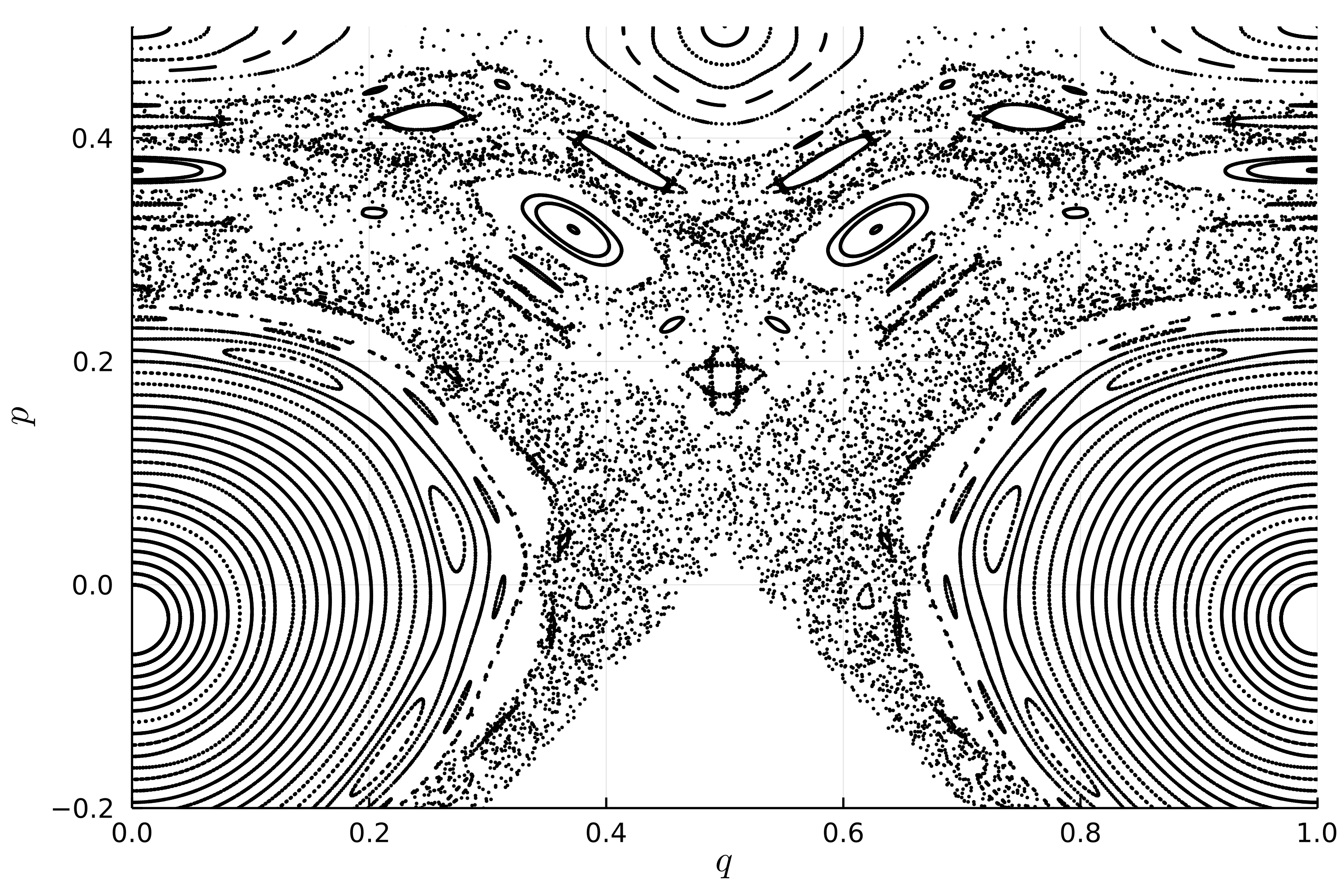}
	\caption{\footnotesize Poincar\'{e} section for \Eq{TwoWaveHam} at $t = 0 \mod 1$ 
	$\mu = 0.03$.}
	\label{fig:twoWavePoincare}
\end{figure}

\subsubsection{Distinguishing Regular and Chaotic Orbits}
To demonstrate the difference in convergence of the WBA between chaotic and regular orbits we consider the weighted Birkhoff average of the function $h(q,p,t) = p$. Note that the average of this function for any quasi-periodic orbit will be $\rho$, the rotation number of the orbit.
\beq{RotNum}
	\rho = \lim_{T \to \infty} \frac{q(T)-q(0)}{T} 
	     = \lim_{T \to \infty}\frac{1}{T} \int_0^T p(\tau)d\tau 
	     = \langle p \rangle,
\eeq
taking the lift of the coordinate $q$ to $\R$.

Figure~\ref{fig:digVST_DW} shows the maximum digit accuracy \Eq{maxdig} as a function of $T$ for two initial conditions,
$(q_0,p_0) = (0,0.45)$ and $(0,0.3)$. As can be seen in \Fig{twoWavePoincare}, the first orbit lies
 on a librational torus in a period-two island, while the second appears to be chaotic. 
For the regular orbit, the WBA appears to converge to ten digits by $T=1000$,
and $\dig_T$ indicates double precision accuracy of $\langle p \rangle$ by $T= 2000$. Since the lower bound of $\dig_T$ increases
linearly with $T$, the convergence appears to be exponential. By contrast, $\dig_T$ fluctuates
around $2$ for the second, chaotic orbit. A similar dichotomy was also see for maps in \cite{Sander20}.

\begin{figure}[ht]
	\centering
	\includegraphics[width=0.8\linewidth]{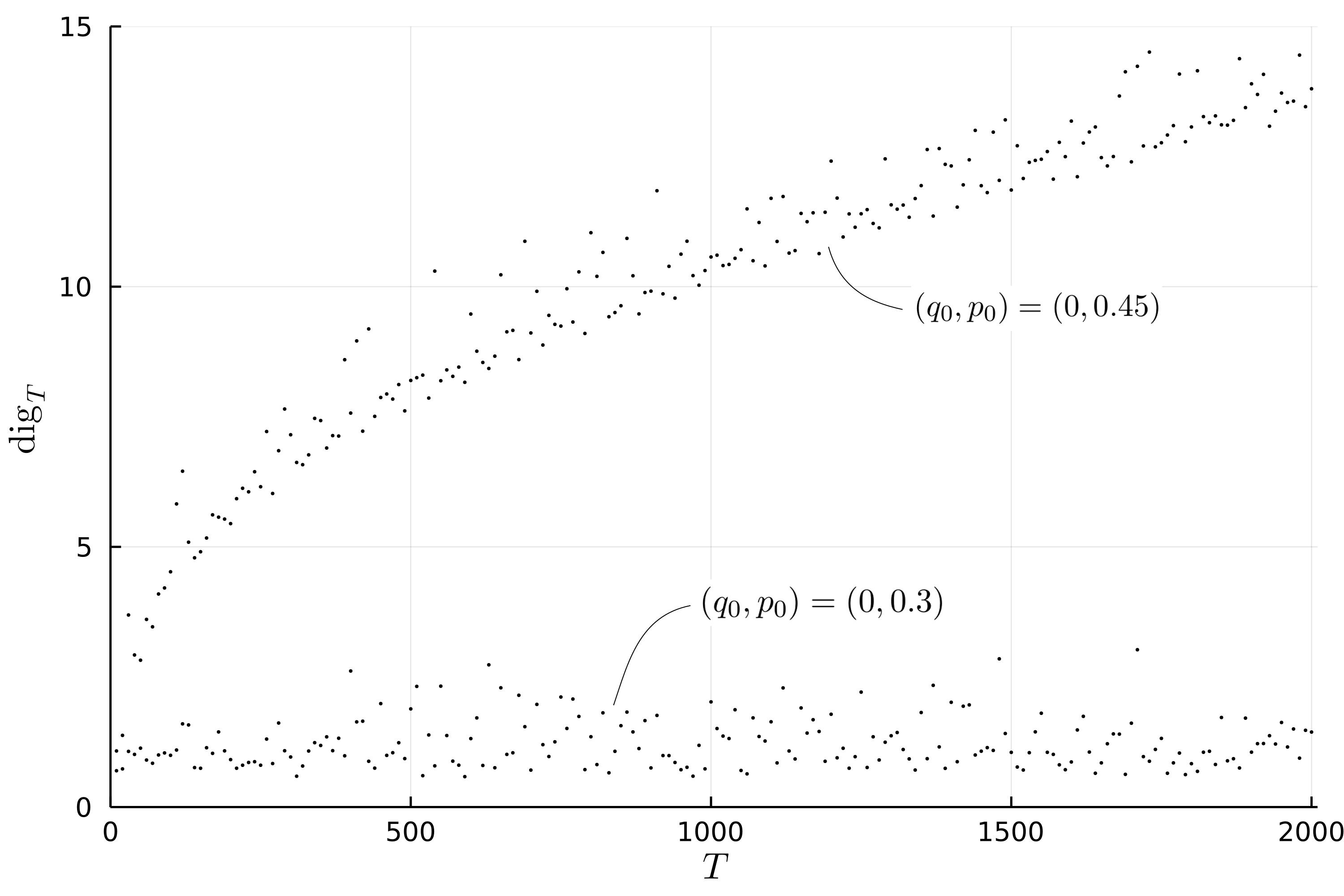}
	\caption{\footnotesize The maximum digit accuracy $\dig_T$ vs $T$ using the function $h(q,p,t) = p$ for two orbits 
	of \Eq{TwoWaveHam} with $\mu = 0.3$. The orbit with initial condition $(q,p) = (0,0.45)$ is regular 
	and that with initial condition $(q,p)=(0,0.3)$ is chaotic.}
	\label{fig:digVST_DW}
\end{figure}

To reinforce the distinction between convergence rates for chaotic and regular orbits, \Fig{PoincarePlotContinuous_10000} shows
a heat map of the maximum digit accuracy for $T= 1000$ for $501$ initial conditions $(0,p_0,0)$, $0 \le p_0 \le 0.5$ on the Poincar\'{e} section of \Fig{twoWavePoincare}. Note that the regular orbits in the low-period islands have $\dig_T \gtrsim 10$, while the strongly chaotic orbits outside these islands have $\dig_T \lesssim 3$. Hence, the colors show that there is a clear distinction between the regular and chaotic orbits in \Fig{twoWavePoincare}.
 
\begin{figure}[ht]
	\centering
	\includegraphics[width = \linewidth]{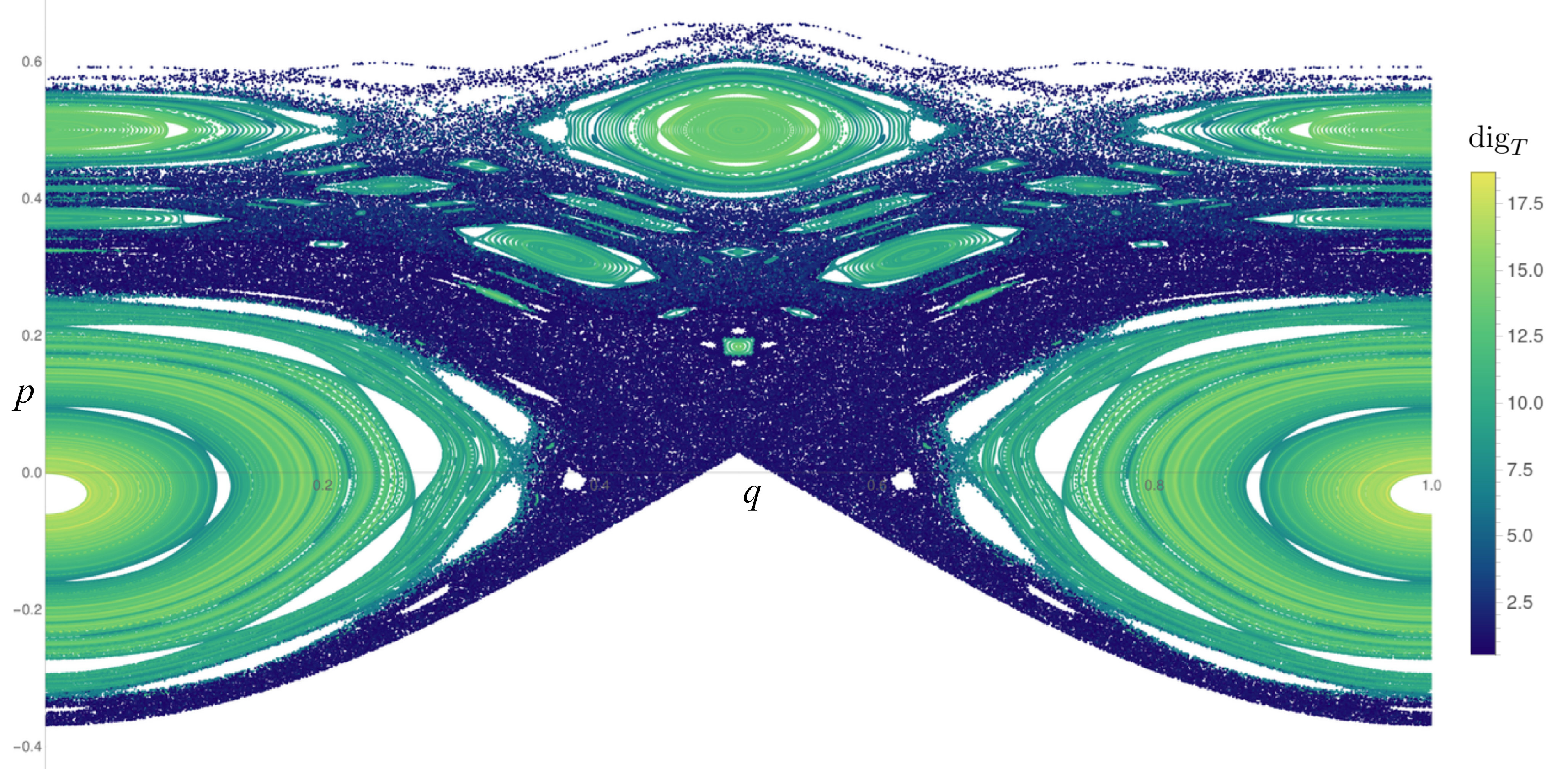}
	\caption{\footnotesize A Poincar\'e section of \Eq{TwoWaveHam} for $\mu = 0.03$.
	The $500$ orbits have initial conditions $(0,p_0,0)$, $p_0 \in [0,0.5]$.
	The colors, as shown in the color bar, represent $\dig_T$ for $T= 1000$ and $h = p$.}
	\label{fig:PoincarePlotContinuous_10000}
\end{figure}

\subsubsection{Comparisons of Relative and Absolute Accuracy}

\Fig{digVSp0_DW} shows the results of computations of the criteria \Eq{absdig} and \Eq{reldig}
for the same set of initial conditions as \Fig{PoincarePlotContinuous_10000}.
Panel (a) shows that $\reldig_T$ performs poorly for the regular, librational tori around $(0,0)$. This is expected since $\langle p \rangle = 0$ for these orbits, so that the denominator of \Eq{reldig} is near zero. However, $\reldig_T$ nears machine precision for the orbits that are trapped in the period-two island chain, near $p = 0.5$.

Panel(b) shows that $\absdig_T$ clearly distinguishes between the regular, island-trapped orbits and the chaotic orbits that were seen in \Fig{twoWavePoincare}. The initial momenta corresponding to chaotic orbits have $\absdig_T \lesssim 3$, while the regular orbits have $\absdig_T \gtrsim 7$, and there are only five orbits with $4 < \dig_T < 6$.
A plot of $\dig_T$, not shown would be identical to panel (b): for this case $\dig_T = \absdig_T$ for all orbits. 

For the two-wave model, the results in \Fig{digVSp0_DW} indicate a threshold for distinguishing chaotic and regular orbits
\beq{ChaosCriterion}
	\dig_{1000} < 5, \qquad \mbox{(chaos criterion)},
\eeq
for the function $h = p$.

\InsertFigTwo{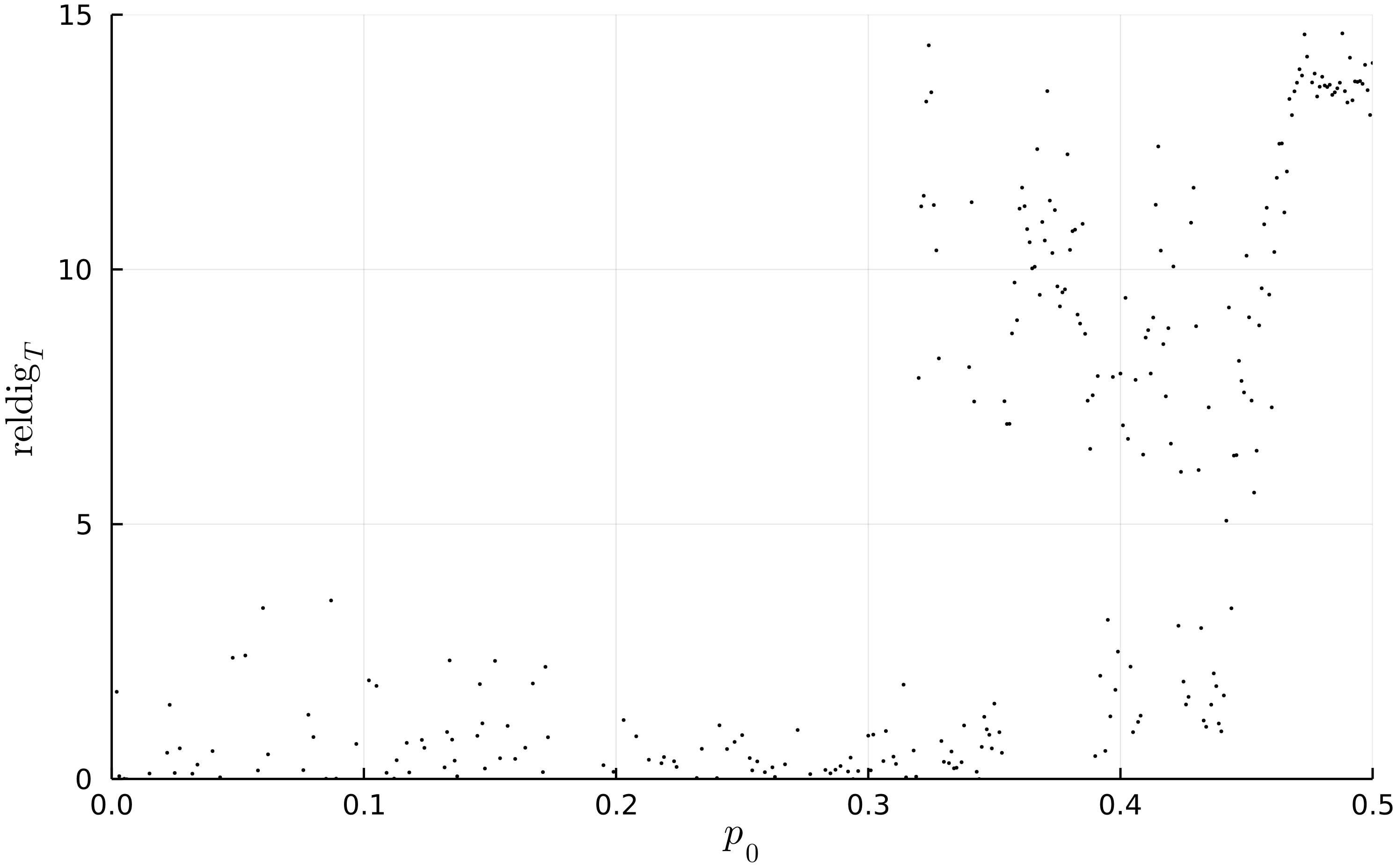}{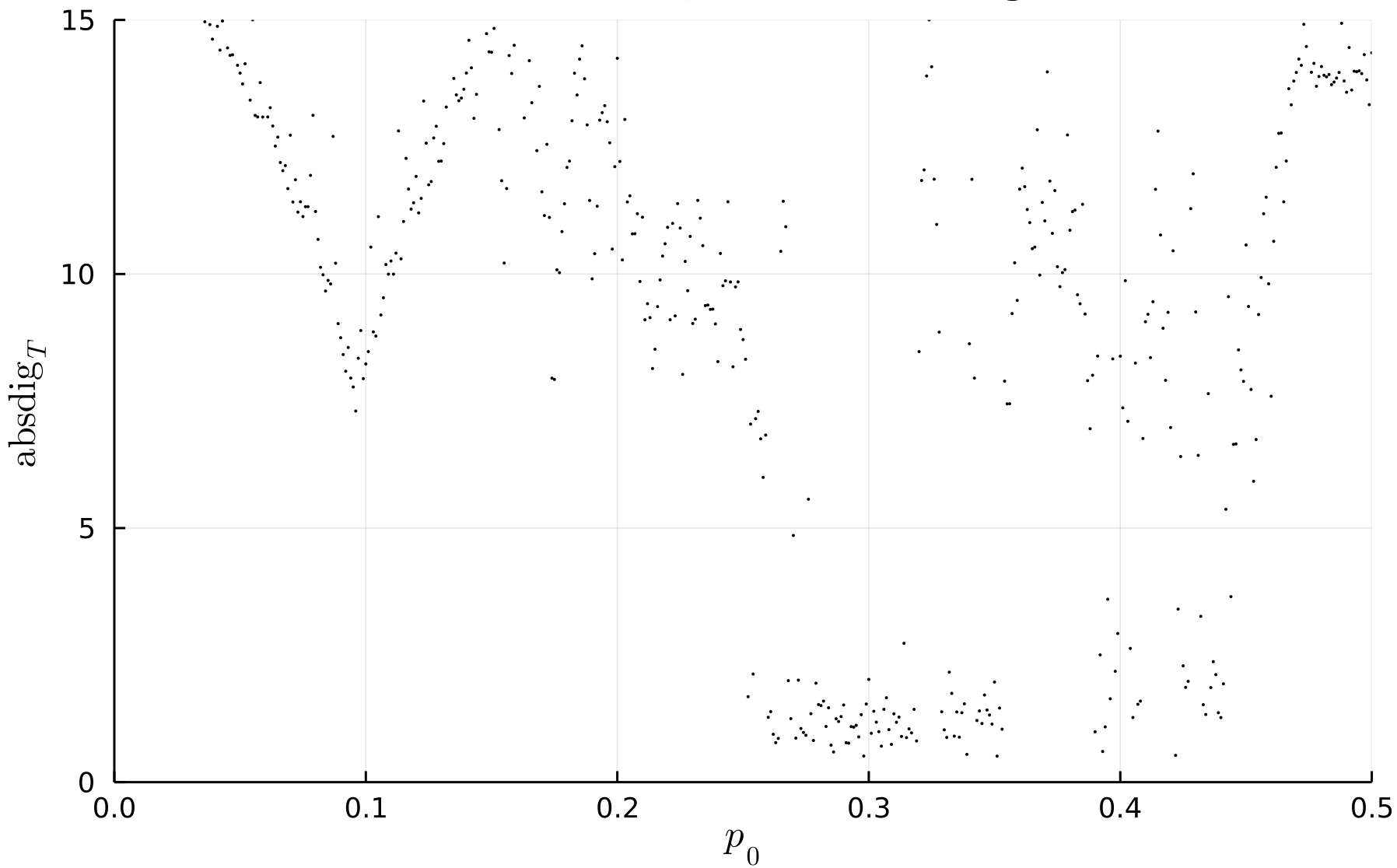}
{Relative and absolute digit accuracy as a function of $p_0$ for 
initial conditions along the line $q_0 = 0$. In each case $T=1000$ and $h=p$.}
{digVSp0_DW}{0.48}

A curiosity of \Fig{digVSp0_DW}(b) is the decrease in $\absdig_T$ near $p_0 = 0.1$, even though most orbits trapped in the period-one island appear to be regular. As can be seen in \Fig{PoincarePlotContinuous_10000}, this dip corresponds to the region near to a hyperbolic period-six
orbit on the section that starts near $(0,0.1)$. These regular orbits lie just outside the separatrix of this island chain and thus take a long time to complete a full rotation---indeed this time would go to infinity for the integrable case as the orbit nears the separatrix. Moreover, the librational rotation number of these orbits around the elliptic point approaches the rational that corresponds to that of the hyperbolic orbit. Consequently such orbits will explore a smaller fraction of an invariant torus over some finite time than those with initial condition further away. The result is a less-accurate, finite-time approximation of the space average of $h$.

\InsertFig{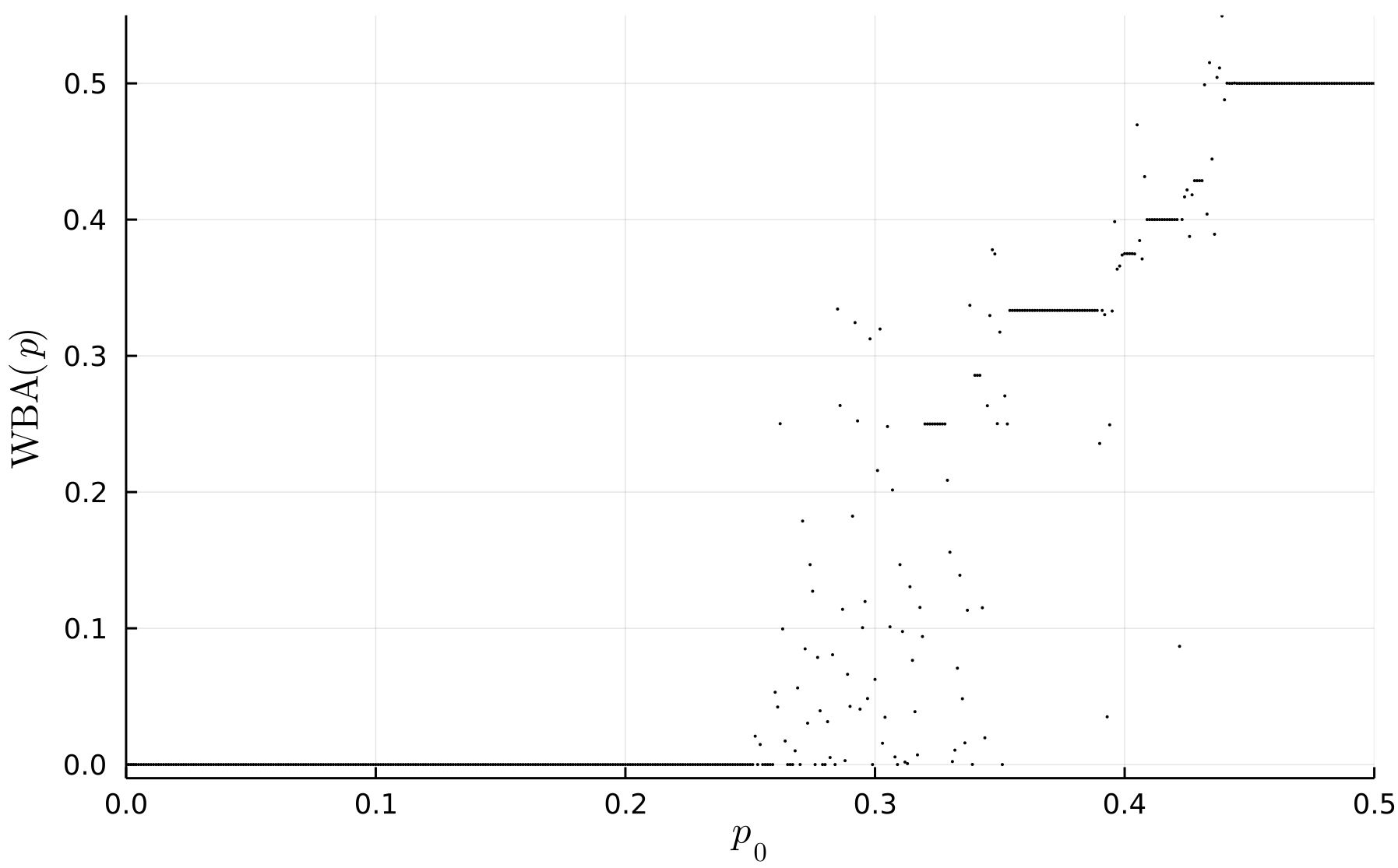}
{The computed rotation number $\rho \approx \WB_T(p)$ for initial conditions $(0,p_0,0)$ as a function of $p_0$. As above this is for \Eq{TwoWaveHam} with $\mu = 0.03$.}
{WBA_DW}{0.7}

The computed rotation number \Eq{RotNum} as a function of $p_0$ is shown in \Fig{WBA_DW}. In the region of librational tori near $(0,0)$, $\rho=0$. Since most of these orbits are regular $\rho$ is computed with high accuracy. Indeed, as we saw in \Fig{digVSp0_DW}, $\absdig_T>5$ for all $p_0 \in [0,0.251]$.
The rapid fluctuations in $\rho$ as a function of initial condition near $p_0 = 0.3$ reflect the poor convergence of the WBA for these chaotic orbits. Additional regular regions appear for
higher period islands around elliptic periodic orbits; these have constant, rational rotation number. 
The chaotic regions between pairs of neighboring islands give the scattered values between the flat intervals.

\subsubsection{Varying the Width of the Bump Function}

Another choice worth investigating is that of the weight function $g:\R\to[0,\infty)$ in \Eq{WeightedBirkhoff}. \Th{superconvergence} implies that super-convergence follows whenever $g$ is $C^\infty$ and flat at $0$ and $1$. In this section we continue to use \Eq{Bump}, but vary it slightly by adding a width parameter, $w>0$:
\beq{WidthModG}
	g_{w}(s) = 
	\begin{cases}
		C \exp\left( \frac{-w}{s(1-s)}\right), & s\in (0,1) \\
		0, 		& s\leq 0 \mbox{ or } s \geq 1.
	\end{cases}
\eeq
Again, $C$ is chosen so that $g_{w}$ has the normalization \Eq{Normalization}.
The resulting function is shown for several values of $w$ in \Fig{WeightedBump}. 
If $w \ll 1$ then $g_w$ is near its maximum over a large fraction of $[0,1]$, and the average
limits to the unweighted, time-$T$ average. If $w \gg 1$ then $g_w$ is essentially zero except
for a small interval. Neither of these cases would seem to be desirable. But what intermediate value of $w$ is best?

\InsertFig{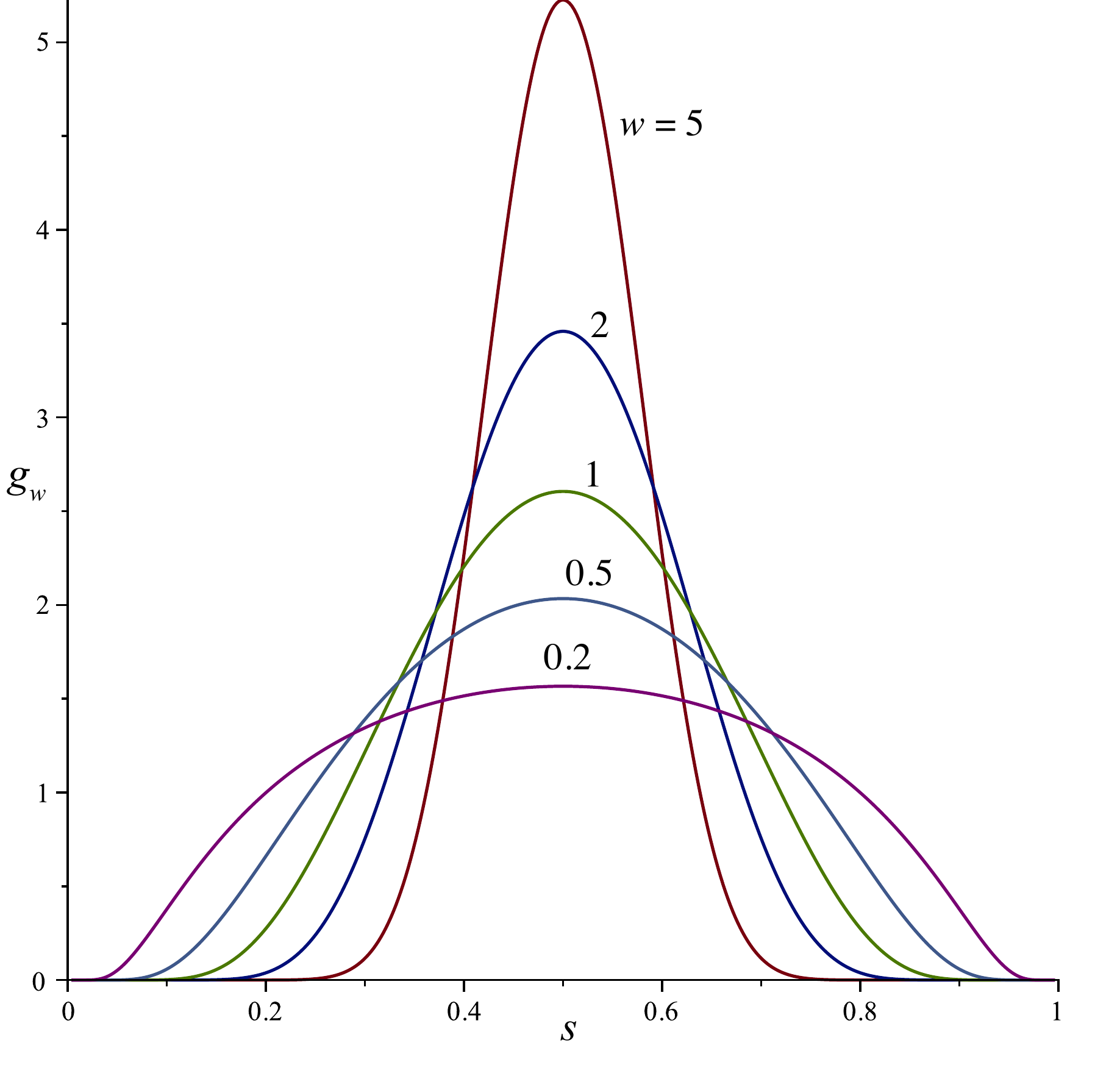}
{The weighted bump function \Eq{WidthModG} for five values of $w$}
{WeightedBump}{0.4}

Figure~\ref{fig:VaryWidths}, shows how $\dig_T$ depends on $w$ for two different $T$,
for a regular trajectory of the two-wave model.
Interestingly, these curves have local spikes indicating improved convergence for nearly isolated values of $w$. 
A possible reason is that these choices ensure that the heavily weighted portion in the average corresponds
to an interval width that is approximately an integer multiple of the rotation number
for this orbit. Since, as seen in the figure, the value of $w$ for these spikes changes with $T$,
and would also change with rotation number of the torus, it is hard to argue that such a 
choice for $w$ would be optimal. In any case, $w = 1$ seems to be a reasonable choice, 
since---if we ignore the spikes---the accuracy has a local maximum near this point. 

\InsertFigTwo{digVSw0_labelled_mod}{digVSw0_30_labelled_mod}
{The variation of $\dig_T(h)(z)$ with $w$, the width parameter of \Eq{WidthModG}. The system is \Eq{TwoWaveHam} with $\mu = 0.03$, $h = p$, and initial condition $(q_0,p_0) = (0,0.1)$. (a) $w \in [0,1]$ and (b) $ w \in [1,50]$.}
{VaryWidths}{0.5}

\subsection{A Quasiperiodically Forced System}\label{sec:QPendulum}

As \Th{superconvergence} showed, the weighted Birkhoff average is super-convergent for a
quasiperiodic orbit with Diophantine rotation vector. When the dynamics is a conjugate to rigid rotation,
then there is an invariant measure on the torus. More generally, if there is no invariant measure,
then the Birkhoff ergodic theorem does not apply. In this case it is not clear whether the accuracy
of the weighted Birkhoff average would be able to distinguish between regular and non-regular orbits.

In this subsection we study the quasiperiodically forced and damped pendulum model of \cite{Grebogi84, Romeiras87}:
\beq{QPODEs}
	\begin{aligned}
		\dot{\theta} &= p ,\\
		\dot{\psi}_1 &= \gamma ,\\
		\dot{\psi}_2 &= 1 ,\\
		\dot{p}   &= -\nu p + a \cos(2\pi \theta) + b + c (\cos(2\pi \psi_1) + \cos(2\pi \psi_2)).
	\end{aligned}
\eeq
Here $\gamma \in \R\setminus \Q$ is irrational and $a,b,c,\nu \in \R$ are parameters.
Grebogi et al.~\cite{Grebogi84} observed that this system can have a geometrically
strange attractor with no positive Lyapunov exponents, a situation that they call a
\textit{strange, nonchaotic attractor}. Even though such a system may be thought of as nonchaotic because nearby
orbits do not separate exponentially, the dynamics may still
exhibit the weaker, topological form of sensitive dependence \cite{Glendinning06}.

Formally the phase space for \Eq{QPODEs} is $\T^3 \times \R$, with coordinates $(\theta,\psi_1,\psi_2,p)$.
A natural 3D Poincar\'e section is $\psi_2 = 0 \mod 1$.
Following \cite{Romeiras87} we take 
\beq{QPValues}
	\nu = a = 6\pi, \quad c = 0.55\nu, \quad \gamma = \tfrac12(-1+\sqrt{5})
\eeq
leaving one free parameter, $b$.\footnote
{In \cite{Romeiras87} the parameters are $\nu = a = 2\pi p$, $b = 2\pi K p$, and $c=2\pi V p$, where $p$ is a damping parameter. The case $p=3$, $K=1.33$, and $V=0.55$ of \cite[Fig.~5]{Romeiras87} corresponds to \Eq{QPValues} with $b = 1.33\nu$.}

Six examples of attractors for \Eq{QPODEs} are shown in \Fig{QPSections}. These figures are projections onto $(\theta,p)$ of the Poincar\'e section $\psi_1 = 0 \mod 1$. On the 3D section the attractor is sometimes a two-torus, sometimes geometrically strange with dimension between two and three, and sometimes fully 3D. For example, when $b = 1.33 \nu$ the system has a geometrically strange attractor with a box-counting dimension larger than $1$, but no positive Lyapunov exponents: in the $(\theta,p)$ subspace the maximal exponent is $\lambda = -0.45$ \cite{Romeiras87}.
The maximal Lyapunov exponent reaches $0$ at about $b = 1.66\nu$ where the attractor in the Poincar\'e section appears to be 3D, see the final panel of \Fig{QPSections}.
The attractor collapses back to a two-torus as $b$ nears $1.8\nu$ (not shown).

\InsertFigSix{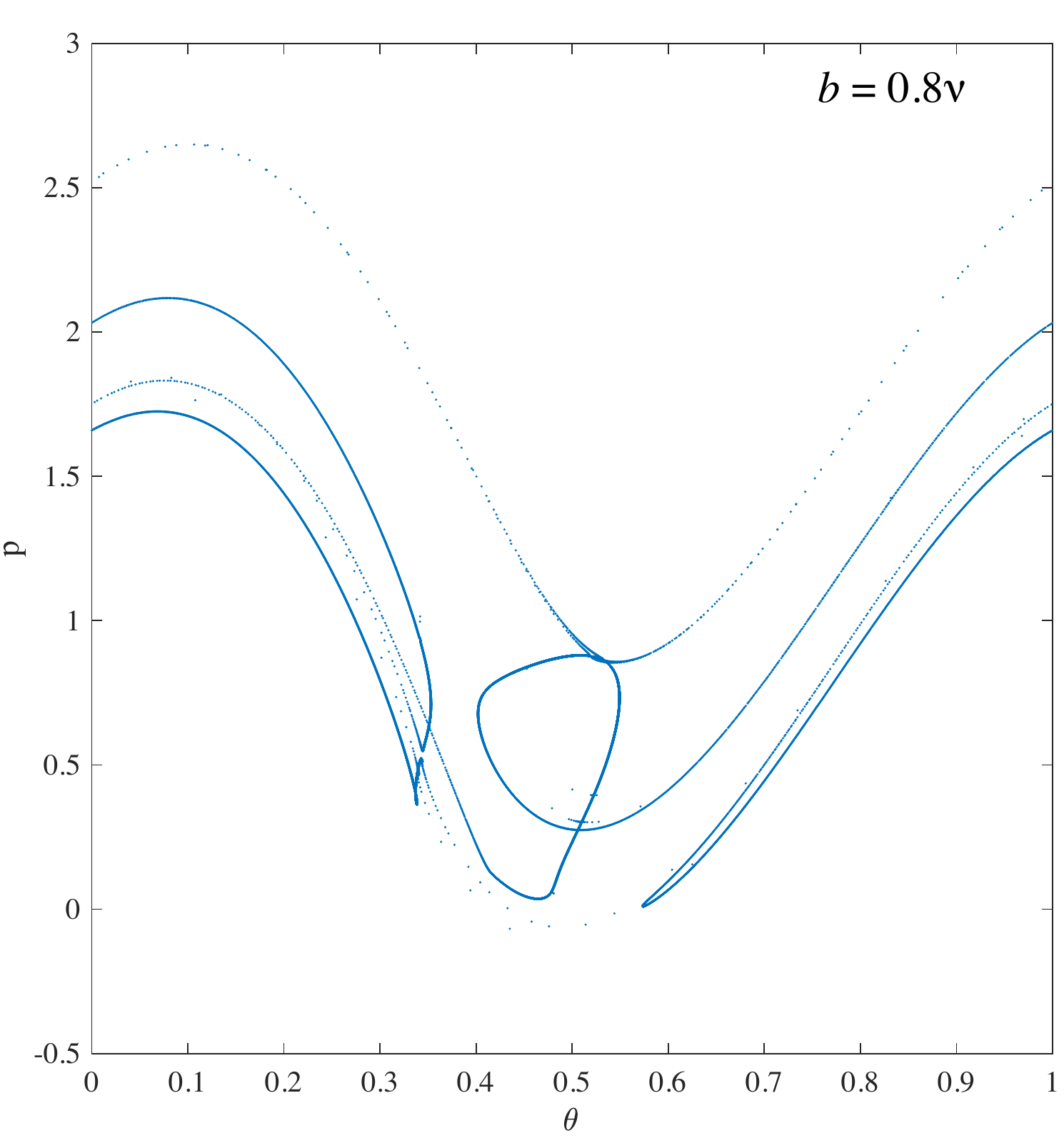}{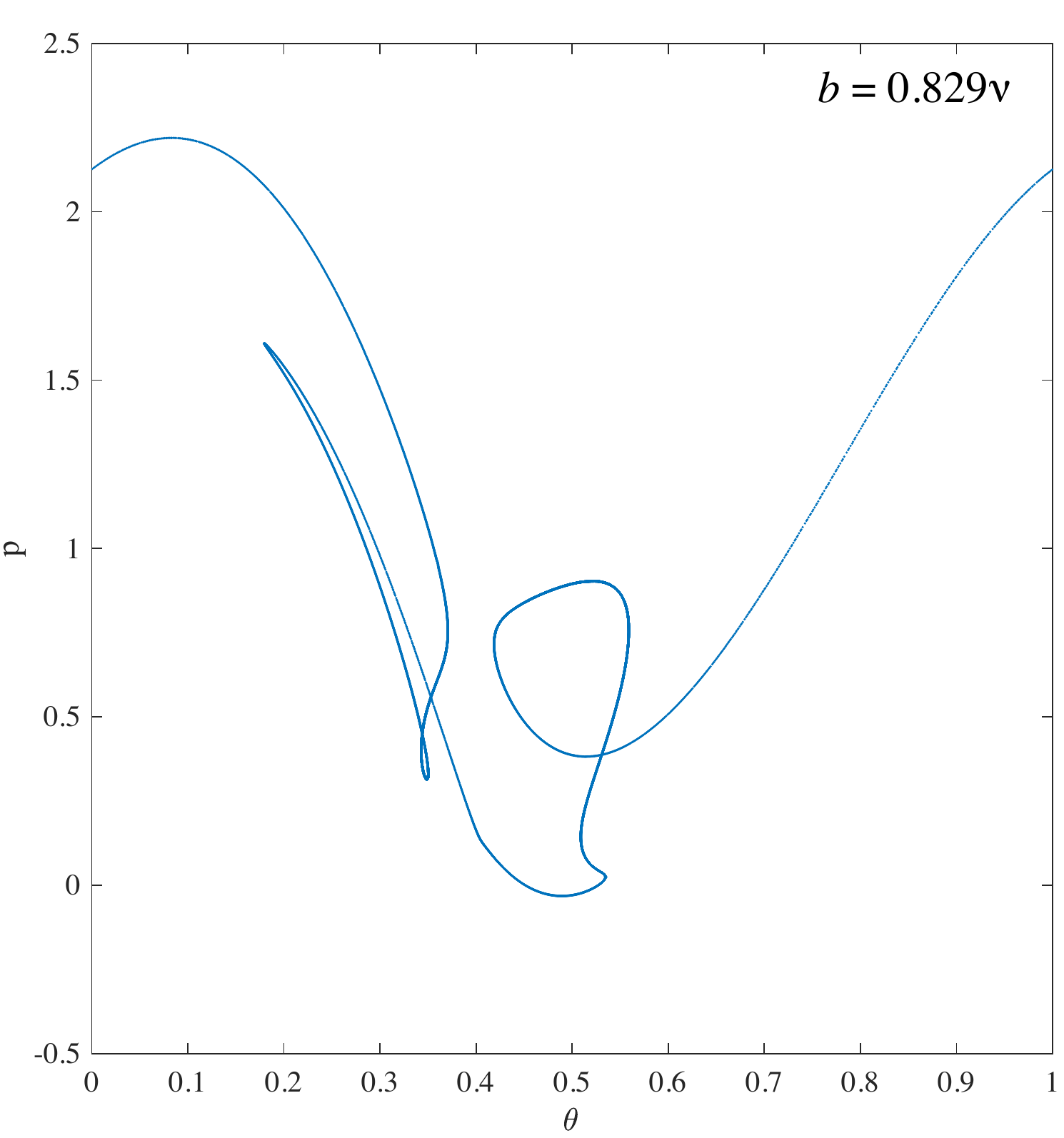}{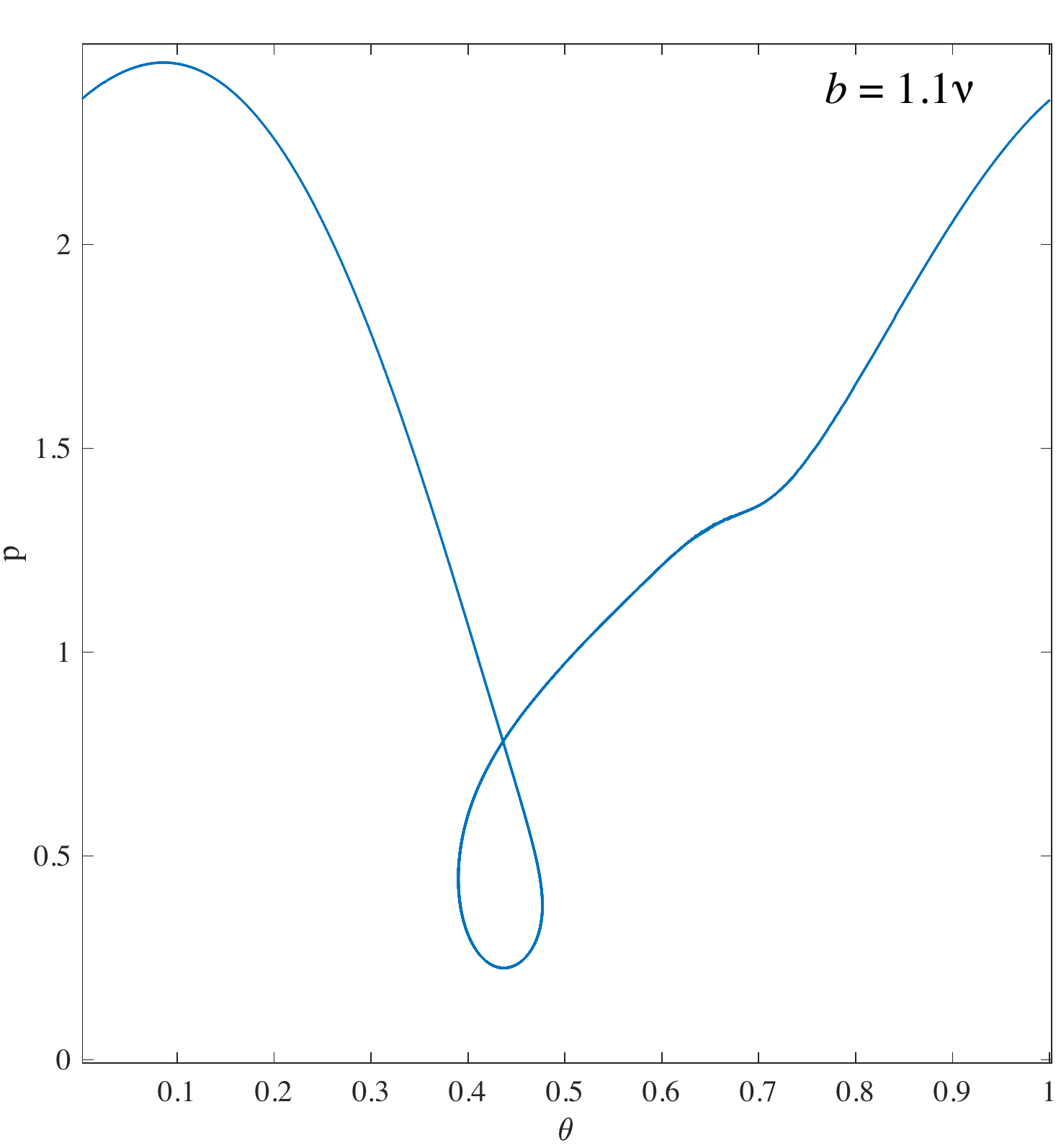}{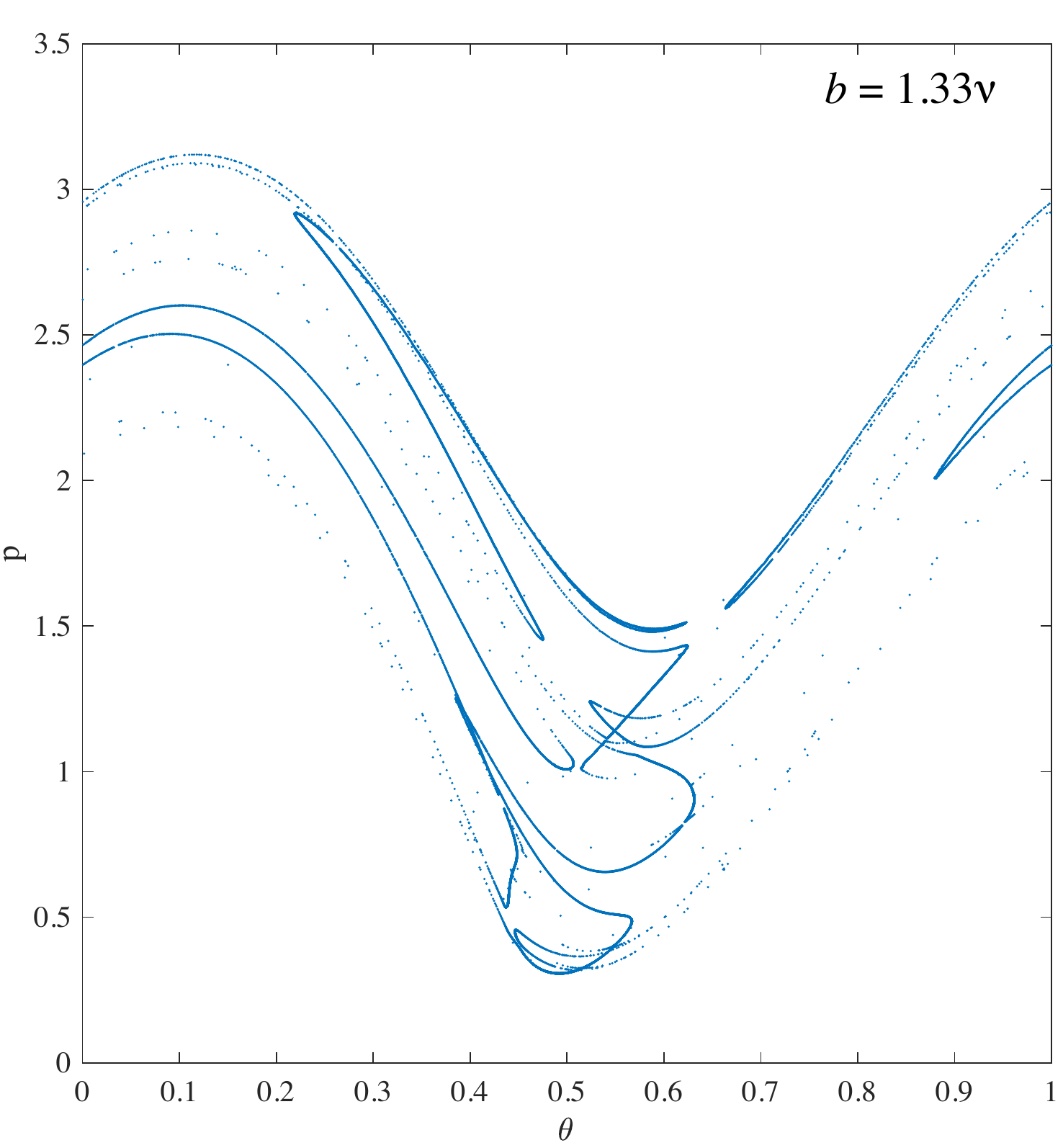}{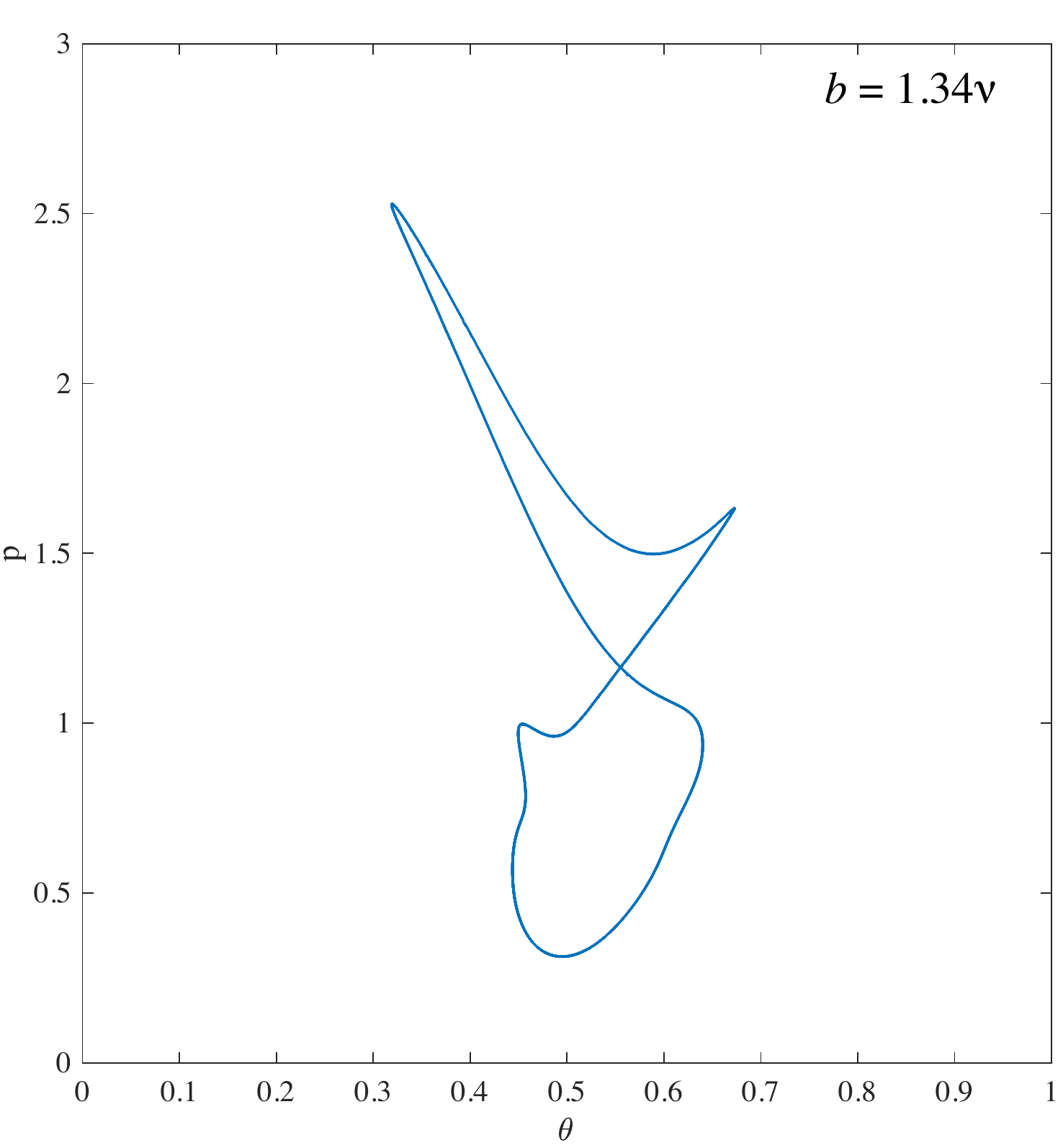}{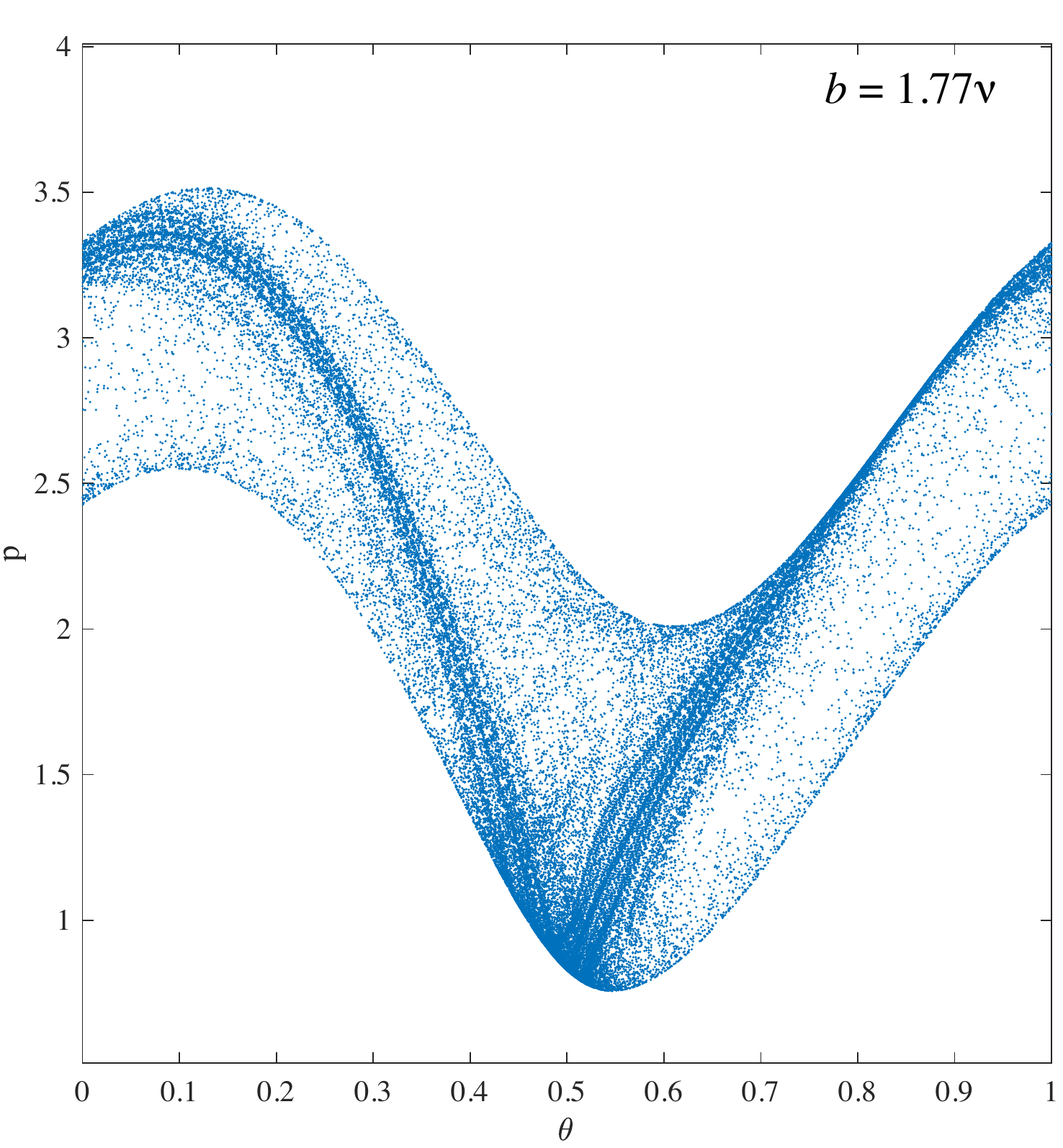}
{Projections onto $(\theta, p)$ of $5\times 10^4$ points on 
	the 3D Poincar\'e section $\psi_2 = 0 \mod 1$ of the system \Eq{QPODEs}
	using \Eq{QPValues}. The six panels show 
	$b = [0.8,0.829,1.1,1.33,1.34,1.77]\nu$, respectively.
	The maximal Lyapunov exponent in the $(q,p)$ subspace is negative, except for the last case, where 
	attractor appears to be 3D on the Poincar\'e section.}
	{QPSections}{2.5 in}

\subsubsection{Distinguishing Strange Attractors}

The different geometric structures shown in \Fig{QPSections} make this system a prime candidate
for investigating whether weighted Birkhoff averaging can be used to distinguish between regular
and strange nonchaotic attractors.
Figure~\ref{fig:digVST_QP} shows $\dig_T$ as a function of $T$ for values of $b$ that 
correspond two-torus, strange, and 3D attractors, respectively.

\begin{figure}[ht]
	\centering
	\includegraphics[width=0.8\linewidth]{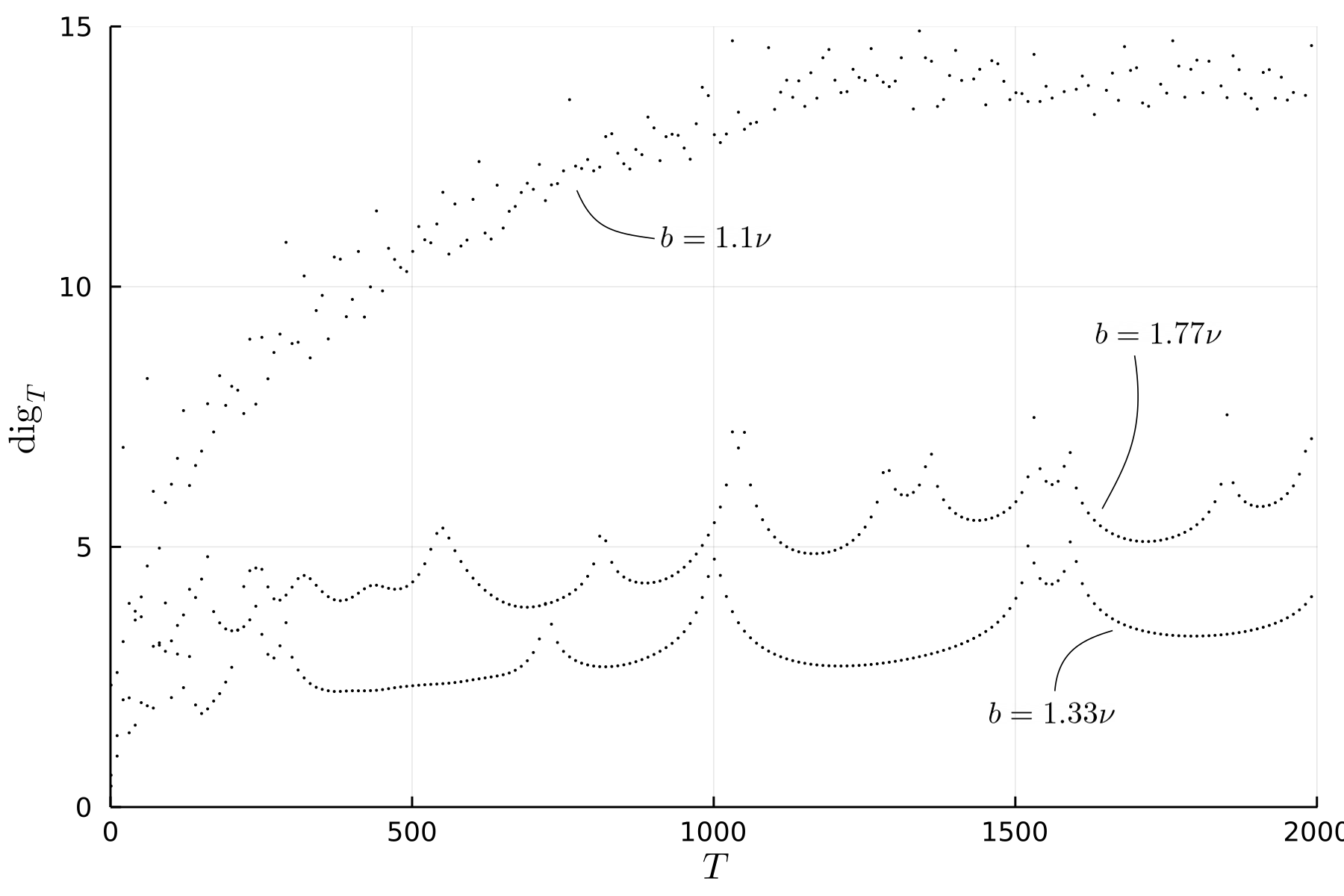}
	\caption{\footnotesize The digit accuracy for the system \Eq{QPODEs} with parameters \Eq{QPValues} and $b = 1.1\nu$, $1.33\nu$, and $1.77\nu$ as a function of $T$ using $h = p$. The initial condition is $(\theta,\psi_1,\psi_2,p)=(0,0,0,2)$.}
	\label{fig:digVST_QP}
\end{figure}

When $b = 1.1\nu$, the WBA appears to be super-convergent. The maximum digit accuracy reaches $13$ by $T=1200$ and then remains nearly constant; this is consistent with the accuracy of the numerical integration that was set to $10^{-13}$ for both absolute and relative error.
This suggests that the dynamics of this orbit are conjugate to a Diophantine rigid rotation. 
When $b = 1.33\nu$, where the attractor is strange but nonchaotic, the convergence of the WBA in \Fig{digVST_QP}
is observed to be poor: it only reaches $3$ by $T=1200$.
The convergence is also poor when $b= 1.77\nu$, where the attractor is 3D.
Even though $dig_T$ is larger than the previous case, it only reaches $5$ when $T = 1200$, and
in both cases the WBA appears to converge---at best---at a polynomial rate in $T$. 
Even if this attractor is simply a three-torus, the relatively poor convergence suggests that
its dynamics are more complex than rigid rotation.
Thus the weighted Birkhoff average effectively distinguishes between a two-torus attractor,
and more complex or higher dimensional attractors.

\subsubsection{Finding Two-Tori}
We now look at how the accuracy of the WBA varies with $b$ in order to distinguish between 
two-torus and strange or 3D attractors.
Figure \ref{fig:digVSb_QP} shows $\dig_T $ as a function of $b \in [0.6,1.8]\nu$.
The figure shows a clear stratification into three levels; the highest corresponds to $13 < \dig_T <18$. 
This high accuracy occurs, for example, for the case $b = 1.1\nu$ and $1.34\nu$ shown in 
\Fig{QPSections} that are clearly two-tori; these are the blue points in \Fig{digVSb_QP}.
The highest accuracy, $\dig_T \sim 17$, occurs near $b = 0.6\nu$. The attractors in this case
(not shown) are even simpler: they resemble librating orbits of the pendulum in $(\theta, p)$
that are simply extended in the $\psi_1$ direction. 

The lowest level in \Fig{digVSb_QP} are those $b$ values with $\dig_T \sim 4$. The two red points correspond to the values $b = 0.8\nu$, and $1.33\nu$, the strange attractors shown in \Fig{QPSections}. 

The mid-level range, $\dig_T \sim 8$, for \Fig{digVSb_QP} corresponds to geometrically more complex attractors that are nevertheless, not strange. For example, the green points in the figure, represent the values $b = 0.829\nu$ and $1.77\nu$ shown in \Fig{QPSections}. The first appears to be the projection of a two-torus, however, it is geometrically more complex than those tori that have higher values of $\dig_T$. The second green point corresponds to the 3D attractor in \Fig{QPSections}. The rapid increase in $\dig_T$ as $b$ increases beyond $1.77\nu$ in \Fig{digVST_QP} signals the collapse of 3D attractor; by $b=1.8\nu$, it has become a two-torus similar to that at $1.1\nu$ though without the loop seen in
\Fig{QPSections}.

\InsertFig{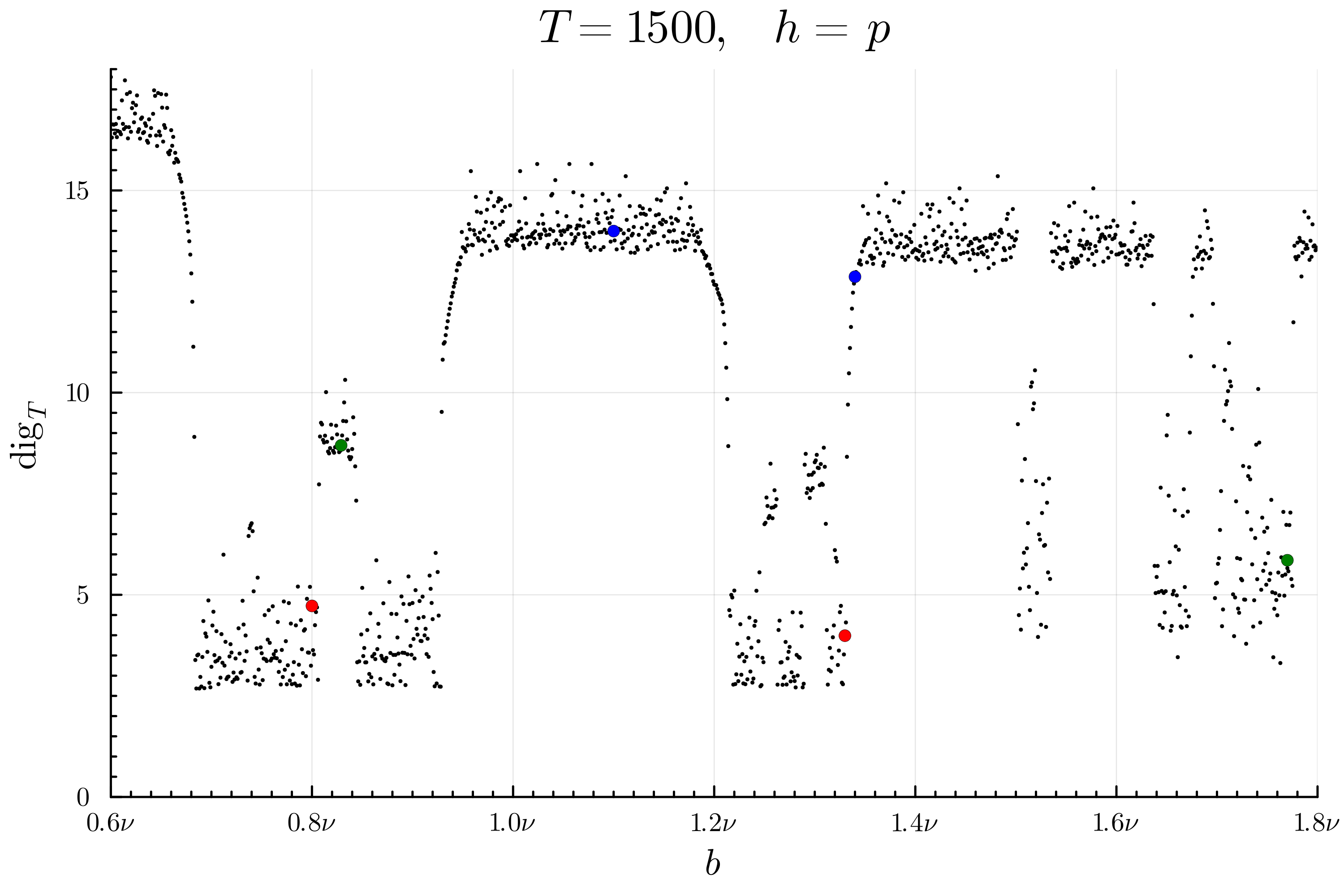}
{The maximum digit accuracy for the system \Eq{QPODEs} with parameters \Eq{QPValues} for
	a grid of $1200$ values of $b \in [0.6,1.8]\nu$.
	Here $T = 1500$ and $h = p$. 
	The red points correspond to $b = 0.8\nu$ and $1.33\nu$, blue to $b= 1.1\nu$, $1.34\nu$,
	and green to $0.829\nu$ and $1.77\nu$. The attractors for these $b$ values were shown in \Fig{QPSections}.}
	{digVSb_QP}{0.7}

The the rotation number of $\theta$:
\[
	\rho = \lim_{n\to\infty} \frac{\theta(n)-\theta(0)}{n},
\]
is shown in \Fig{WBA_QP}. 
The figure is similar to the Devil's staircase shown in \cite[Fig. 8b]{Romeiras87},
however the weighted Birkhoff average for the function $h = p$ provides a much more accurate computation.
The flat sections in $\rho$ correspond to the two-torus attractors with $\dig_T \sim 14$,
the highest level in \Fig{digVSb_QP}. This figure shows almost no scatter when compared
with the corresponding plot for two-wave model, \Fig{WBA_DW}.

\InsertFig{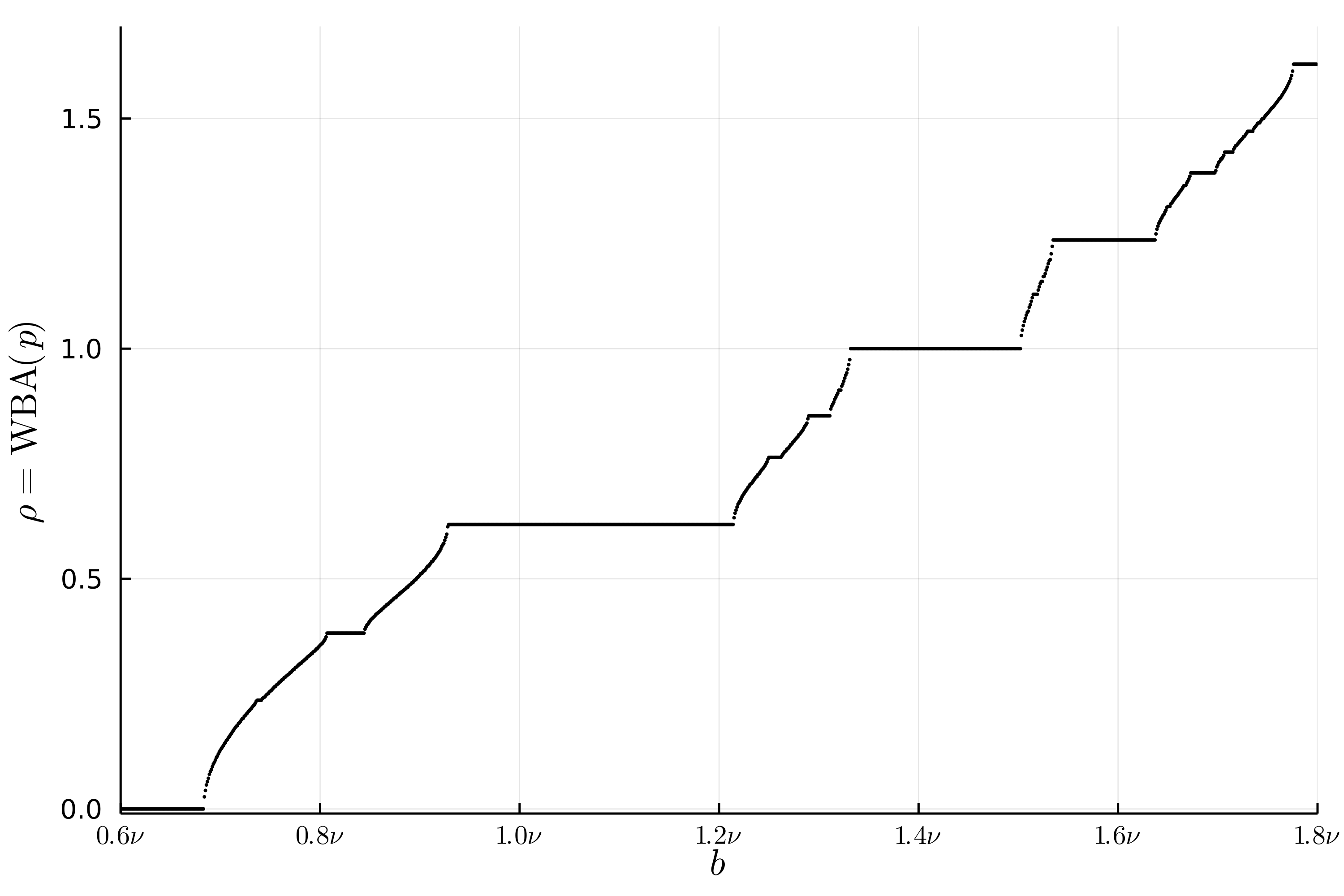}
{The rotation number of $\theta$ for \Eq{QPODEs} for the same parameters as \Fig{digVSb_QP}
computed using the WBA for $h=p$.}
{WBA_QP}{0.7}

\subsection{Magnetic Field Line Flow}\label{sec:FieldLines}

As a final example we consider a family of model magnetic fields studied in \cite{paulHeatConductionIrregular2022}.
Here the domain is the solid torus $D^2 \times \bS^1$, where $ (\psi,\theta) \in [0,1] \times \bS^1 $ are
polar coordinates on the disk $D^2$ and $\zeta$ is the toroidal angle on $\bS^1$.
The fields are generated from the vector potential $A = \psi \nabla \theta - \chi\nabla \zeta$ with
\[ 
	\chi(\psi, \theta, \zeta) = \tfrac12 \psi^2 - 
			\sum_{m,n \in \Z} \eps_{m,n} \psi(\psi-1) \cos\left(2\pi(m\theta - n\zeta)\right),
\]
This gives the magnetic field
\beq{magneticFields}
	B = \nabla \times A = \nabla \psi \times \nabla \theta - \nabla \chi(\psi,\theta,\zeta) \times \nabla \zeta.
\eeq 
The field line of $B$ for this case can also be thought of as the flow of $\chi$ as a nonautonomous
Hamiltonian using $(\theta,\psi)$ as canonical variables and $\zeta$ as ``time''.
These are the solutions to
\beq{FareyFields}
	\begin{aligned}
		\dot{\psi} 	&= -2\pi\sum_{m,n} m \eps_{m,n} \psi(\psi - 1)\sin(2\pi (m\theta - n \zeta)) ,\\
		\dot{\theta}&= \psi - \sum_{m,n} \eps_{m,n}(2\psi - 1) \cos(2\pi(m\theta - n\zeta)) ,\\
		\dot{\zeta} &= 1. 
	\end{aligned}
\eeq
Note that this system has invariant two-tori at $\psi = 0$ and $\psi = 1$. Moreover, when all the amplitudes $\eps_{m,n} = 0$ the system is completely integrable, since $\psi$ is then invariant. More generally, each $m,n$ Fourier mode creates a resonant magnetic island near $\psi = \tfrac{n}{m}$ with amplitude $\eps_{m,n}$.

In \cite{paulHeatConductionIrregular2022}, a set of resonances with fixed $m$ and a range of $n$ values are studied. 
To emphasize the creation of higher-order islands by resonant beating, we instead use a set of Fourier modes that
correspond to resonances \textit{up to} a given level on the Farey tree \cite{meissSymplecticMapsVariational1992}. In particular, we take $(m,n)$ corresponding to the resonances up to level three on the Farey tree with root $(\tfrac01, \tfrac11)$, namely,
\beq{FareySet} 
	(m,n) \in \{(4,1),(3,1),(5,2),(2,1),(5,3),(3,2),(4,3)\}. 
\eeq
Note that the Farey tree naturally generates a set of coprime $(m,n)$ pairs.

We will choose a one-parameter family of the amplitudes so that there is a critical value at
which the Chirikov overlap criterion \cite{meissSymplecticMapsVariational1992} is simultaneously satisfied for each
neighboring pair. The approximate island half-width in $\psi$
for a single Fourier mode in \Eq{FareyFields} can be obtained by neglecting the $\cO(\eps)$ term in
the $\dot{\theta}$ equation. Thus if $\eps_{m,n} \ll 1$ the system is effectively a pendulum in the variables 
$(m\theta-n\zeta, m\psi-n)$. This gives a resonance at $\psi = \tfrac{n}{m}$ with the half-width
\[
	\Delta_{m,n} = 2 \sqrt{\eps_{m,n} \tfrac{n}{m}\left(1-\tfrac{n}{m}\right)} .
\]
Two neighboring resonant islands on the Farey tree then overlap when
\[
	\Delta_{m_1,n_1} + \Delta_{m_2,n_2} = \left|\frac{n_1}{m_1} - \frac{n_2}{m_2}\right| = \frac{1}{m_1 m_2} .
\]
Here the last equality above follows because the two modes are Farey neighbors. 
If we scale the amplitudes as
\beq{FareyFieldParams} 
 	(\eps_{4,1},\,\eps_{3,1},\,\eps_{5,2},\,\eps_{2,1},\,\eps_{5,3},\,\eps_{3,2},\eps_{4,3}) = 
	\frac{\eps}{21600} \left(72,27,25,96,25,27,72\right)
\eeq
then the resonances simultaneously overlap at $\eps = 1$. For this value the system should be 
chaotic, in the sense that the rotational tori between each pair of islands are destroyed. 
Of course, as is well known, the overlap criterion overestimates the \textit{critical} value for the destruction of the KAM tori \cite{Meiss17a}.

\subsubsection{Detecting Chaos}

A Poincar\'e section at $\zeta=0$ for the system \Eq{FareyFields} is shown in \Fig{FareyPoincare}
for four values of $\eps$. This figure shows only the range $\psi \in [0,\tfrac12]$, but for the 
symmetric amplitudes \Eq{FareyFieldParams}, $\eps_{m,n} = \eps_{m,m-n}$,
the phase portrait has the reflection symmetry $\psi \to 1-\psi$.
Thus the dynamics in the interval $\psi \in [\tfrac12,1]$ can be inferred.

When $\eps=0.05$, \Fig{FareyPoincare}(A) shows that almost all orbits lie on tori, 
though there will invariably be small chaotic regions not observed at this scale near the
separatrices of the island chains. When $\eps = 0.25$, a small amount of chaos 
is visible near the separatrices of the period four and five islands. When $\eps = 0.5$ these chaotic regions grow;
however, there are still rotational tori that act as barriers to transport between each of the primary island chains in the set \Eq{FareySet}.
For $\eps = 1.0$, \Fig{FareyPoincare}(D) shows that all of the rotational tori for $\psi$ in the interval $(0.2,0.8)$ have been destroyed, though some low-period islands persist in the sea of chaos. 

The true critical value of $\eps$ can be estimated numerically by looking for an orbit that 
``crosses'' all the resonances.
Starting at $(\theta_0,\psi_0, \zeta_0) = (0.375, 0.27, 0)$, close to the hyperbolic-point of the $(4,1)$ island, we found that the smallest $\eps$ for which $\psi(t) > 0.45$ for some time $t\in[0,10^4]$ is $\eps_{cr} = 0.665$. This is certainly consistent with the phase portraits in \Fig{FareyPoincare}.

Note that the regions of regular tori around $\psi = 0$ and around $\psi = 1$ persist as $\eps$ grows. This is because the tori $\psi = 0$ and $\psi=1$ are invariant, and the Farey island set \Eq{FareySet} does not include any terms below $\tfrac14$ and above $\tfrac34$.

\begin{figure}[ht]
	\centering
	\begin{subfigure}[b]{0.48\textwidth}
		\includegraphics[width = \linewidth]{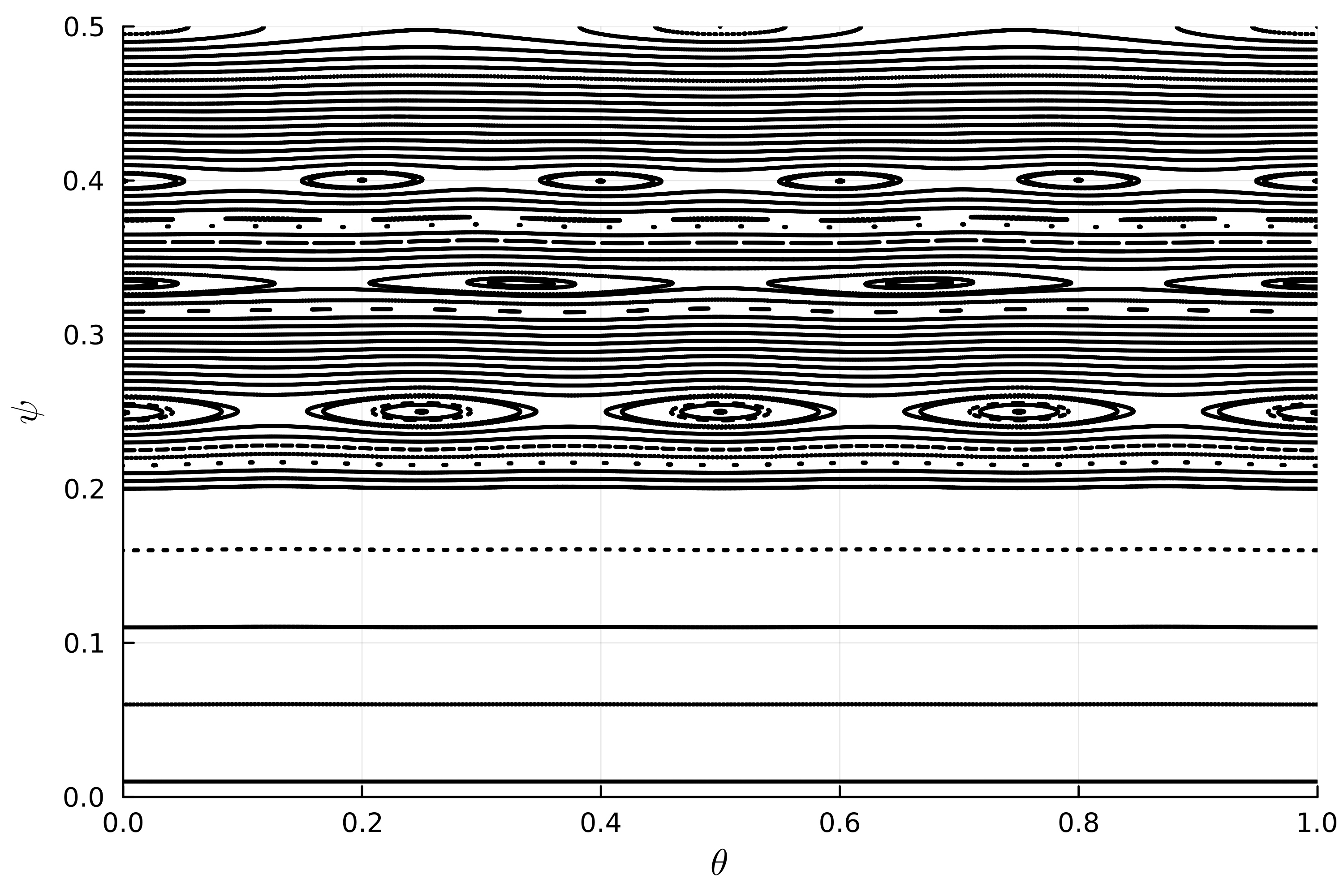}
		\caption{\footnotesize $\eps = 0.05$.}
	\end{subfigure}
	\hfill 
	\begin{subfigure}[b]{0.48\textwidth}
		\includegraphics[width = \linewidth]{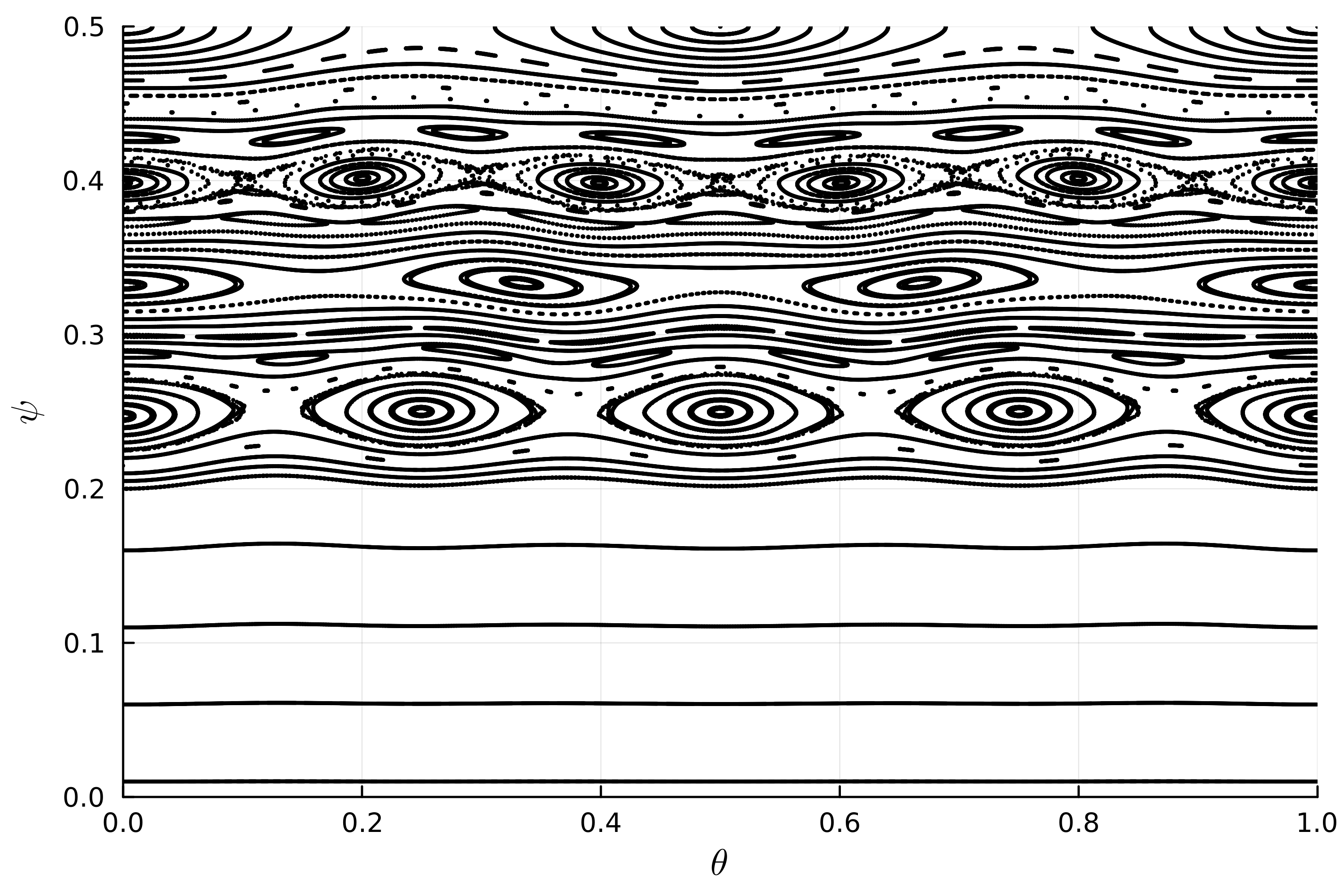}
		\caption{\footnotesize $ \eps = 0.25 $. }
	\end{subfigure}
	\vspace*{10pt}
	
	\begin{subfigure}[b]{0.48\textwidth}
		\includegraphics[width = \linewidth]{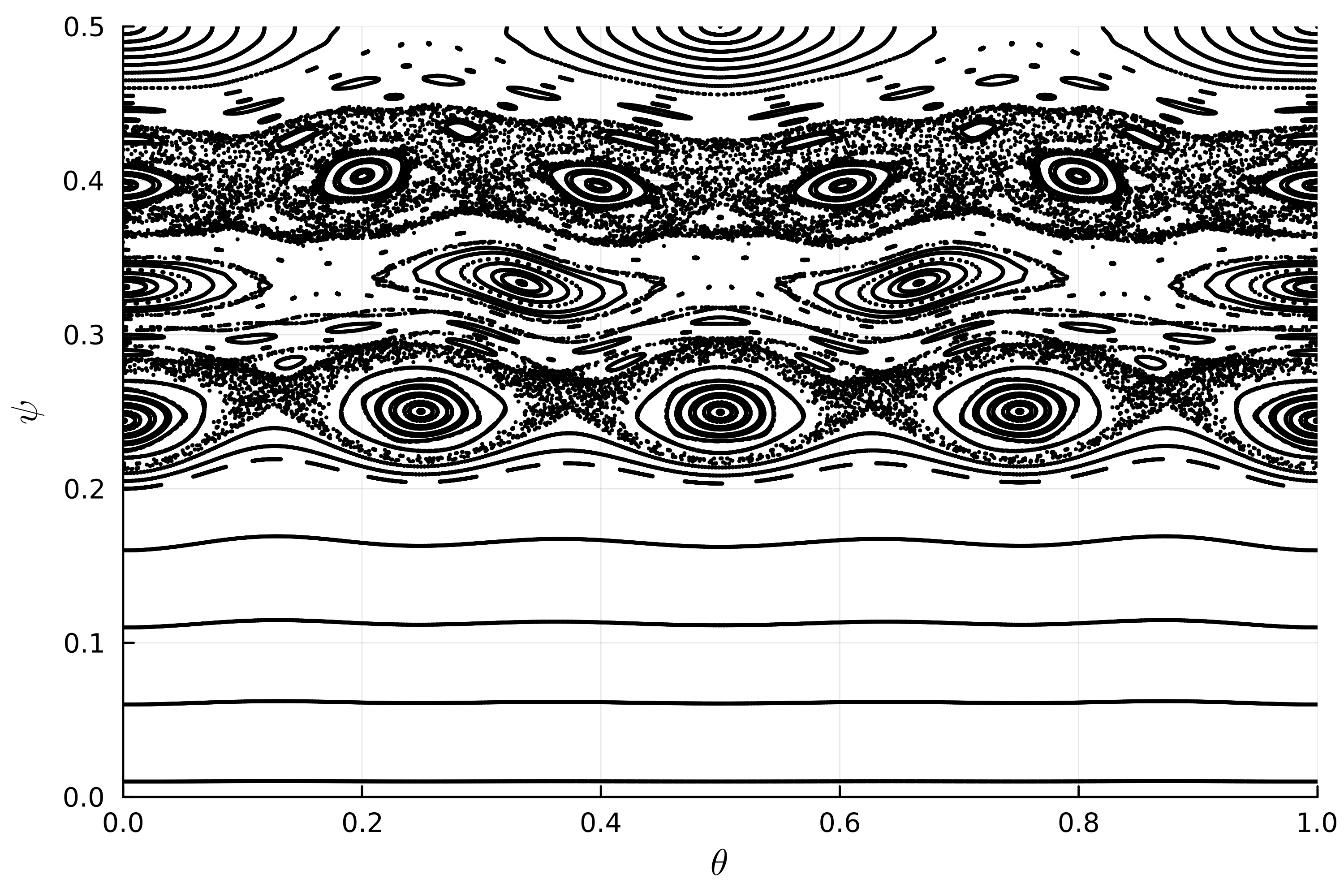}
		\caption{\footnotesize $\eps = 0.5$}
	\end{subfigure}
	\hfill
	\begin{subfigure}[b]{0.48\textwidth}
		\includegraphics[width = \linewidth]{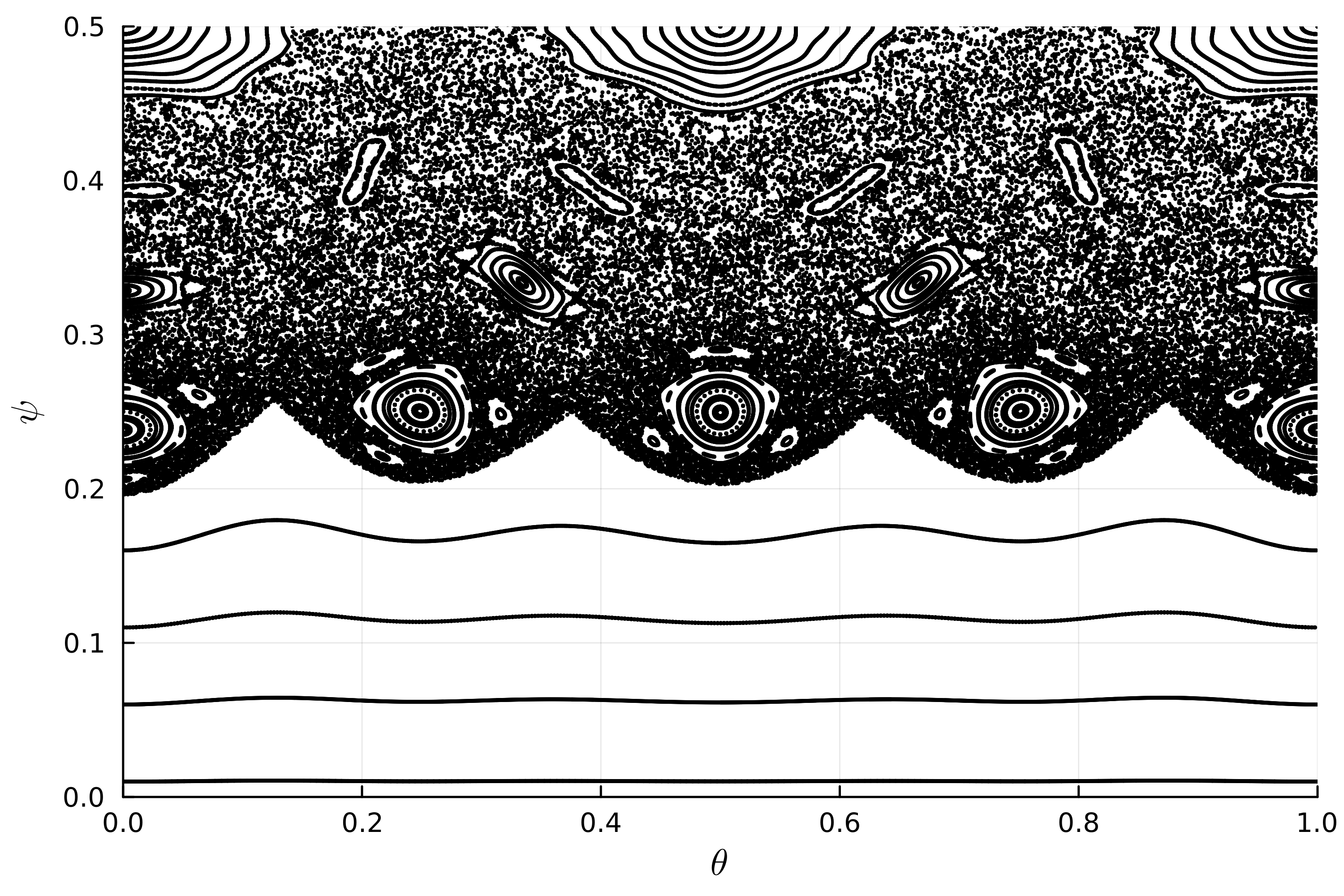}
		\caption{\footnotesize $\eps = 1.0$}
	\end{subfigure}
	\caption{\footnotesize Poincar\'e sections for \Eq{FareyFields} with parameters \Eq{FareyFieldParams} at $\zeta=0 \mod 1$ for for four values of $\eps$.}
	\label{fig:FareyPoincare}
\end{figure}

To study the onset of chaos using the weighted Birkhoff average, we 
computed the maximum digit accuracy, $\dig_T$, for initial conditions
$(\psi_0,0,0)$ with $\psi_0 \in [0,0.5]$. 
We choose $h = \psi$ so that $\langle \psi \rangle$ is the rotation number of a regular torus. 
The results for the same four values of $\eps$ in \Fig{FareyPoincare} are shown in \Fig{digVSp0_FF}. 

\begin{figure}[ht]
	\centering
	\begin{subfigure}[b]{0.48\textwidth}
		\includegraphics[width = \linewidth]{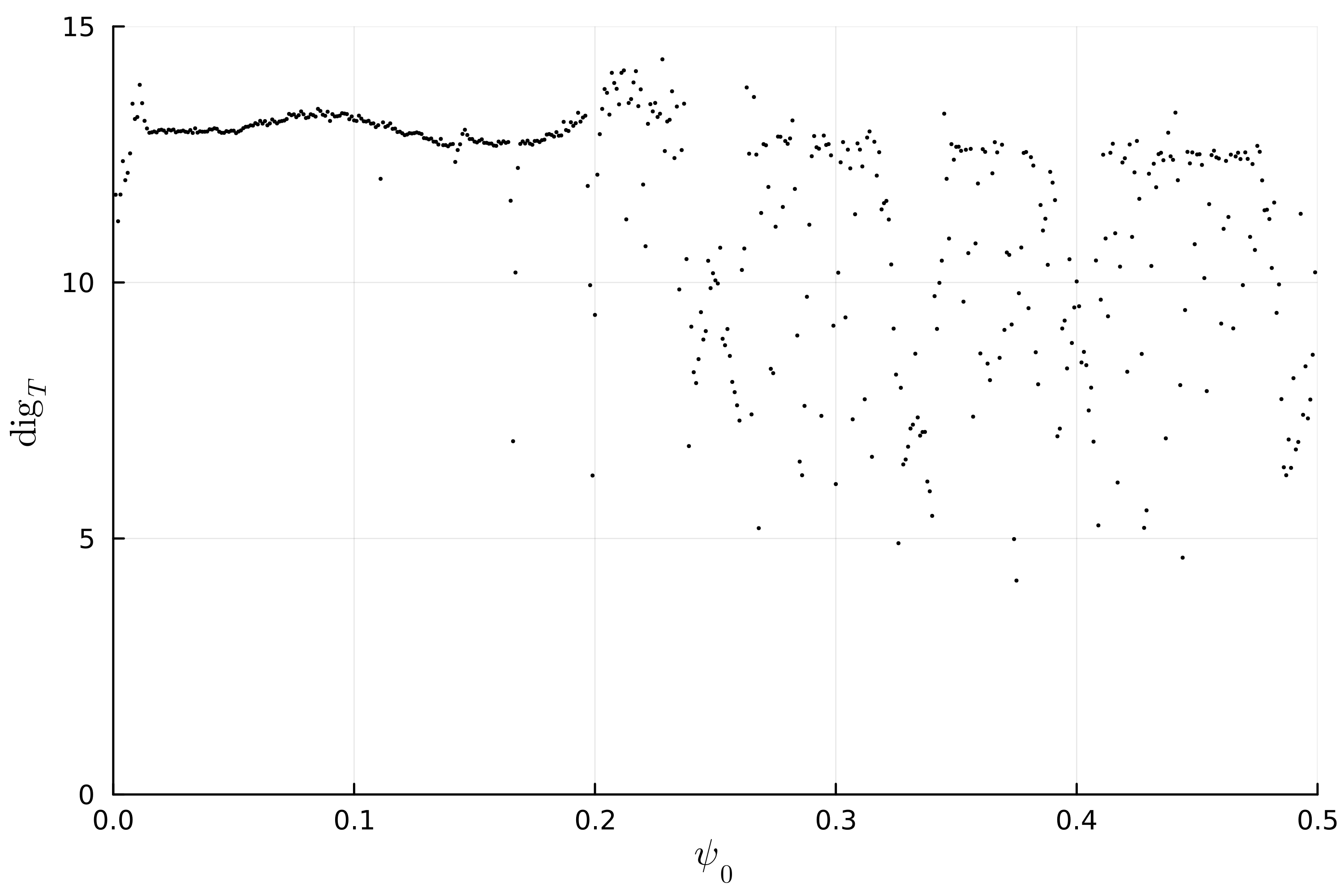}
		\caption{\footnotesize $\eps = 0.05$}
	\end{subfigure}
	\hfill 
	\begin{subfigure}[b]{0.48\textwidth}
		\includegraphics[width = \linewidth]{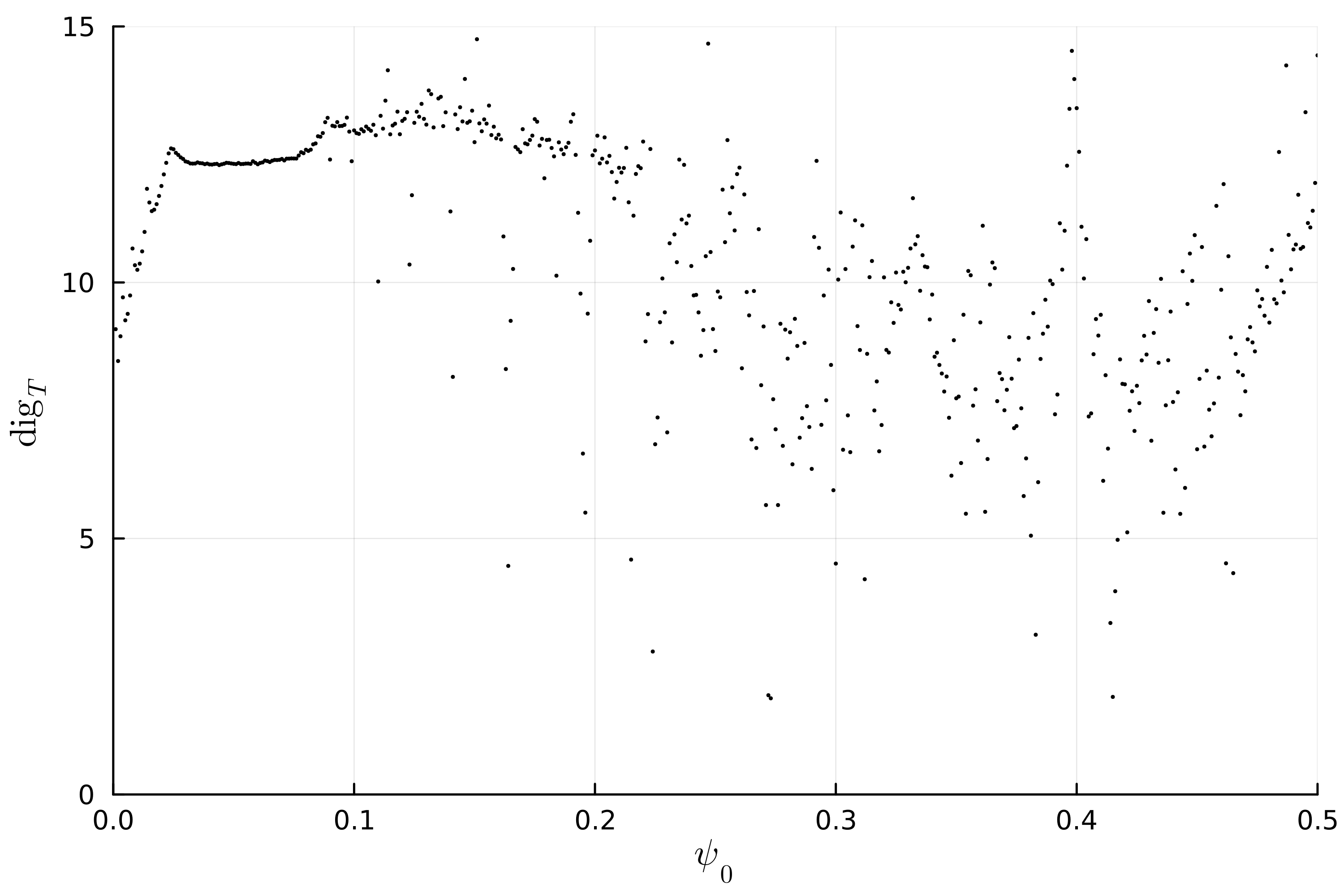}
		\caption{\footnotesize $\eps = 0.25$ }
	\end{subfigure}
	\vspace*{10pt}
	
	\begin{subfigure}[b]{0.48\textwidth}
		\includegraphics[width = \linewidth]{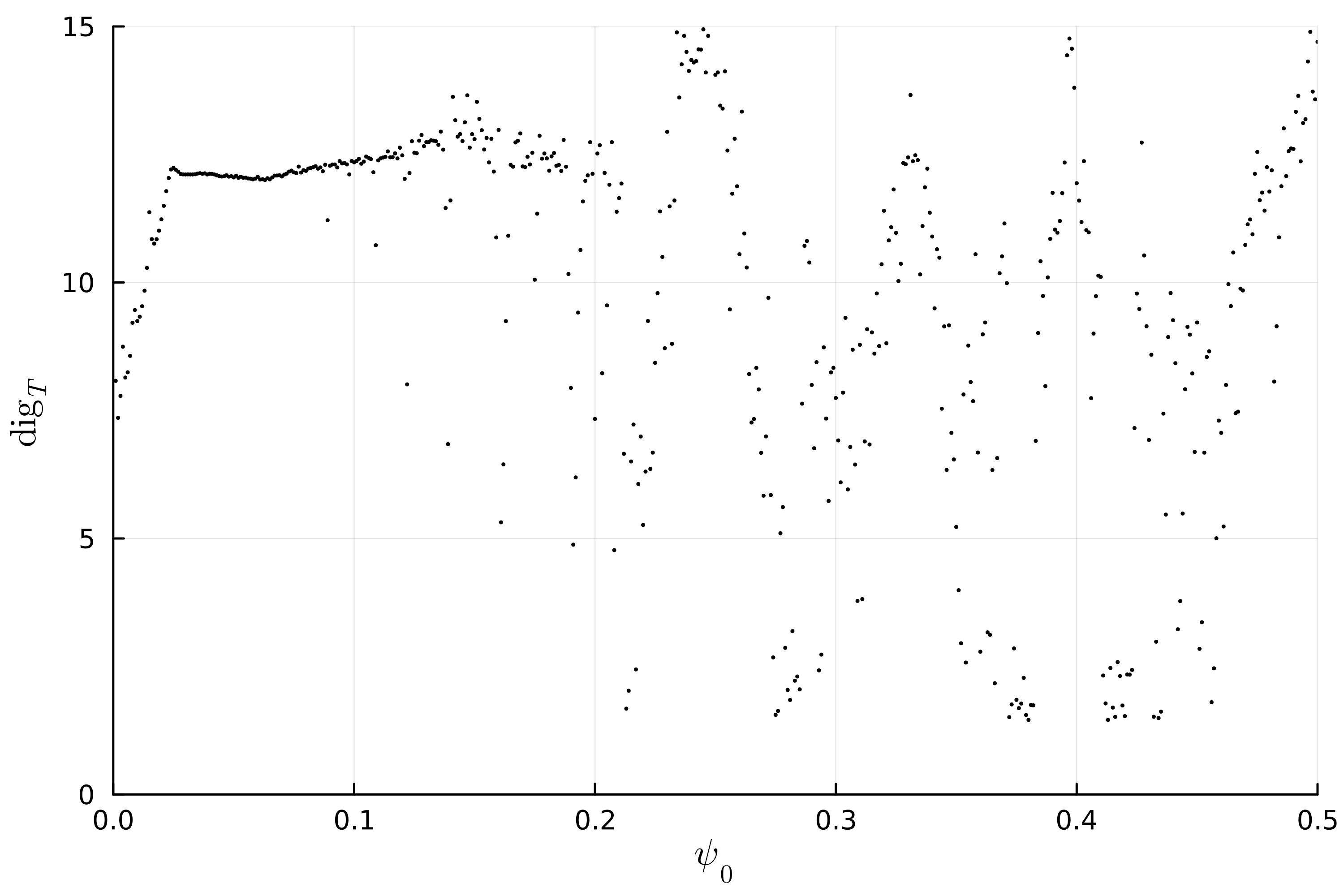}
		\caption{\footnotesize $ \eps = 0.5$}
	\end{subfigure}
	\hfill
	\begin{subfigure}[b]{0.48\textwidth}
		\includegraphics[width = \linewidth]{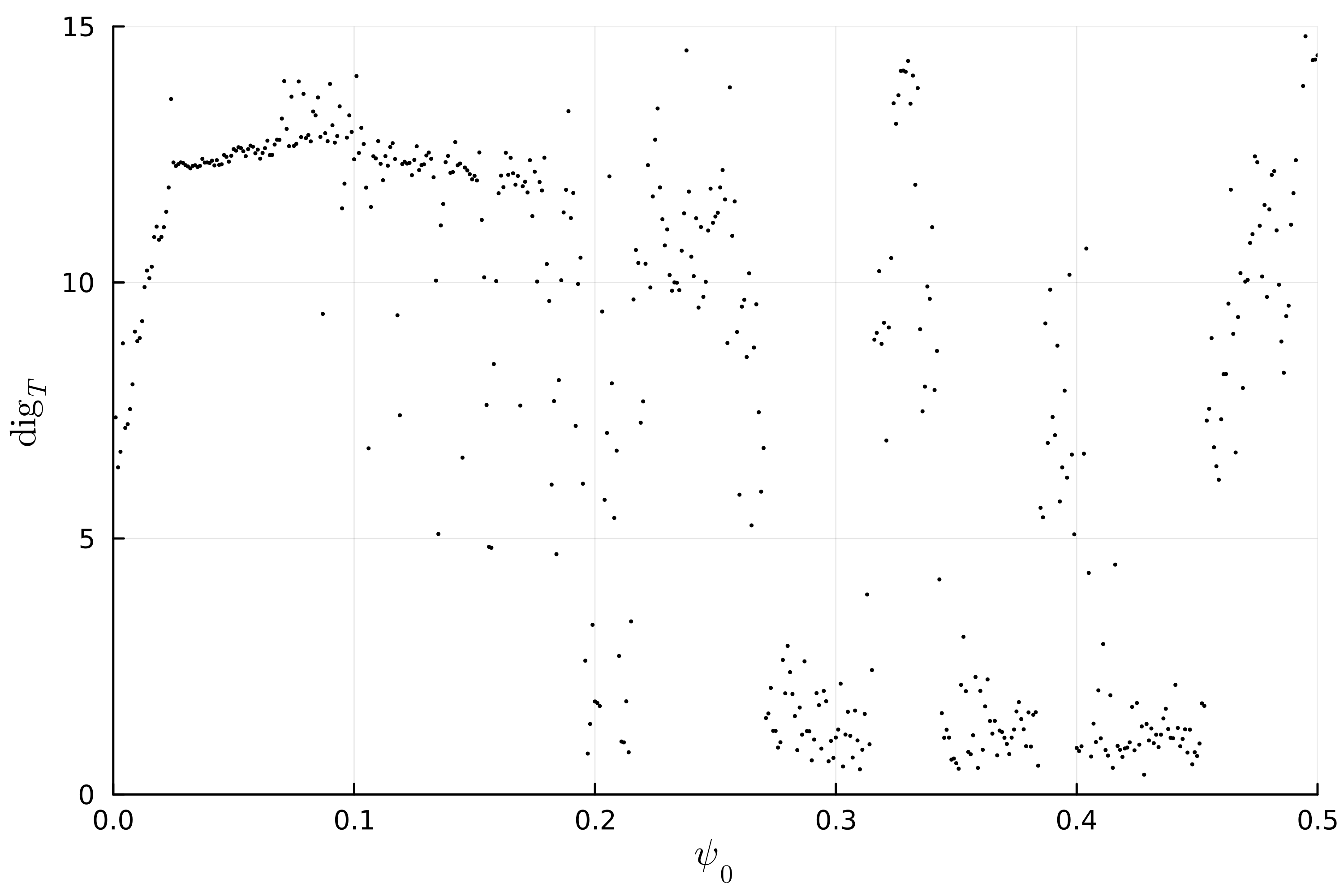}
		\caption{\footnotesize $\eps = 1.0$ }
	\end{subfigure}
	\caption{\footnotesize The maximum digit accuracy, $\dig_T(h)$, as a function of initial condition $\psi_0$ on a
	grid with steps of $0.001$, for four $\eps$ values. Here $T = 1000$ and $h(\psi,\theta,\zeta) = \psi$.}
	\label{fig:digVSp0_FF}
\end{figure}

As can be seen in the figure, the distinction between $\dig_T$ for chaotic and regular orbits becomes clearer as $\eps$ increases.
The criterion \Eq{ChaosCriterion} suggests that chaotic orbits have $\dig_T \lesssim 5$. Note that some of the regular orbits have
$\dig_T$ up to $15$, and that the nearly horizontal, rotational tori in the range $0 < \psi_0 < 0.2$ have $\dig_T \simeq 13$.
When $\eps = 0.05$, there are only $4$ initial conditions on the $501$ point grid that would be designated as chaotic;
though since $4 \le \dig_T \le 5$ for each of these, they are near the threshold.
As $\eps$ increases, chaotic regions surrounding the islands appear and grow. By $\eps = 0.25$,
there are small intervals of low-digit accuracy around the $\tfrac14$ and $\tfrac25$ islands;
however, only $14$ initial conditions have $\dig_T \le 5$.
As $\eps$ increases, the chaotic regions around the low-period islands grow, and for $\eps = 0.5$,
initial conditions with low $\dig_T$ are seen near $\psi_0 = 0.2$, $0.3$, $0.4$ and $0.5$.
Finally, when $\eps = 1.0$ the chaotic regions around the the low period islands merge, though the line of 
initial conditions with $\theta_0 = \zeta_0 =0$ goes through the elliptic center of 
each of the forced resonances in the set \Eq{FareySet}.
This gives rise to the peaks in $\dig_T$ near the low-order islands that are visible in \Fig{FareyPoincare}(D).
Of course, the orbits for $\psi_0 < 0.18$, where regular tori seen in \Fig{FareyPoincare}(D) persist, 
have a maximum digit accuracy that remains high.

To conclude, it is clear by comparing \Fig{FareyPoincare} and \Fig{digVSp0_FF}, that the weighted Birkhoff average accurately detects the onset of chaos for this family of Farey magnetic fields.

\subsubsection{A Measure of Non-integrability}

In \cite{paulHeatConductionIrregular2022}, the authors developed a measure of the ``effective volume of parallel diffusion'' as a proxy for measuring the nonintegrable region. To do this, they first solve the steady-state temperature using the anisotropic diffusion equation,
\beq{ADE}
	\nabla\cdot(\kappa_\parallel\nabla_\parallel T + \kappa_\perp\nabla_\perp T) = 0 .
\eeq
Here $T$ is the temperature, $\kappa_\parallel, \kappa_\perp$ are the parallel, perpendicular diffusion coefficients, respectively, 
and $\nabla_\parallel = \hat{b}\hat{b} \cdot \nabla$ and $\nabla_\perp$ are the gradients parallel
and perpendicular to the magnetic field $B$, respectively. 
Equation \Eq{ADE} is solved with the Dirichlet boundary conditions that fix $T$ on the boundary tori, $\psi = 0$ and $1$.
The measure they use is the fraction of the volume $\Omega$ in which the local parallel heat transport is larger than the perpendicular transport:
\beq{effectiveNI}
	\cV_{PD} = \frac{1}{\operatorname{Vol}(\Omega)}\displaystyle \int_{\Omega} \Theta\left(\kappa_\parallel |\nabla_{\parallel} T|^2 - \kappa_\perp |\nabla_{\perp} T|^2 \right) dx^3.
\eeq
Here $\Theta$ is the Heaviside step function.

In order that \Eq{effectiveNI} be an effective measure of integrability, Paul et. al. \cite{paulHeatConductionIrregular2022} argue that when $\kappa_\perp \ll \kappa_\parallel$, $T$ is approximately constant along field lines of $B$. It follows, for a region foliated by invariant tori, $|\nabla_\parallel T|$ will be relatively small, and thus the measure $\cV_{PD}$ will be essentially zero. 
Conversely, they argue that regions of phase-space with chaotic field lines will have relatively large parallel diffusion. This second claim is shown by first proving that surfaces of constant temperature must have the same topology as the boundary surfaces. Consequently, these isotherms will not be able to completely align to the structure of the field in chaotic regions, increasing the value of $|\nabla_\parallel T|$. 
However, as the authors remark, the effective parallel diffusion will also be high within islands, even if they are not chaotic. Thus regions where the invariant surfaces do not have the same topology as the boundary
will also contribute to \Eq{effectiveNI}. 
In this regard, the measure of parallel diffusion is analogous to converse KAM theory \cite{MacKay89a, Duignan20}.

Regardless of what \Eq{effectiveNI} precisely measures, such a measure may in fact be more useful to the original purposes of \cite{paulHeatConductionIrregular2022} in optimizing the structure of magnetic fields for plasma confinement.
However, the weighted Birkhoff averaging may provide a reasonable alternative measure of chaos if that is desired. 

A simple measure of integrability is the relative fraction of initial conditions deemed chaotic by weighted Birkhoff averaging. For the Farey magnetic fields we first used the same initial initial conditions as in \Fig{digVSp0_FF}: $(\psi_0,\theta_0,\zeta_0) = (\psi_0,0,0)$. 
For each $\eps \in [0,1.0]$ in steps of $0.01$, we computed $\dig_T$ for $T = 1000$ and $h = \psi$.
An orbit was deemed chaotic if $\dig_T < 5$ and regular otherwise. 
The relative fraction of chaotic initial conditions for each $\eps$ is shown using the symbol `+' in \Fig{crudeMeasure}.
A similar computation was performed along the line $\theta_0 = 0.15$.
The result points are shown in \Fig{crudeMeasure} using the symbol $\bullet$.

\begin{figure}[ht]
	\centering
	\includegraphics[width=0.7\linewidth]{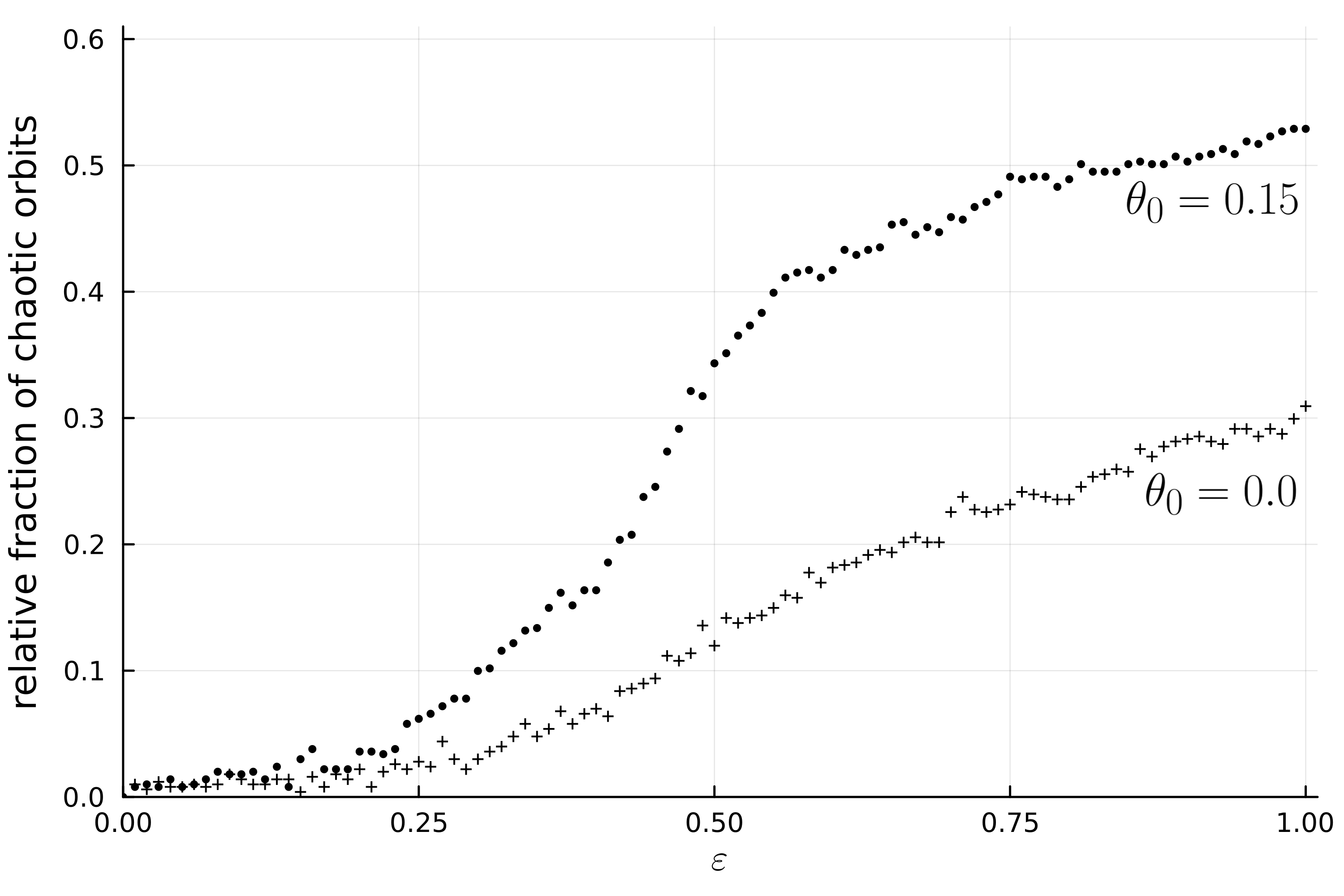}
	\caption{\footnotesize The fraction of initial conditions that are chaotic for \Eq{FareyFields} with $\eps\in[0,1.0]$,
	and initial conditions $(\psi_0,\theta_0,0)$ with $\psi_0 \in[0,0.5]$ in steps of $ 0.001 $. }
	\label{fig:crudeMeasure}
\end{figure}

For both sets of initial conditions, the fraction of chaotic initial conditions is observed to vanish when $\eps = 0$, and increase---though not always monotonically--- with $\eps$. This is consistent with \Fig{FareyPoincare}. For larger $\eps$, the rate of increase of the chaotic fraction slows. This is most prominent for $\theta_0 = 0.15$ near $\eps = 0.6$, just below $\eps_{cr} = 0.665$ where the last rotational tori in the interval $\psi_0 \in [0.18,0.5]$ are destroyed. 

An issue with the use of the fraction of chaotic orbits along a line of initial conditions in the 3D phase space is evident in the difference between the two cases in \Fig{FareyPoincare}. The line
$\theta_0= 0$ is special because it intersects the rotational periodic orbits at their elliptic points; this
is evidently due to a time-reversal symmetry of the system \Eq{FareyFields} under the involution
\[
	(\psi,\theta,\zeta ) \to (\psi,-\theta,-\zeta).
\]
A similar time-reversal symmetry of the Chirikov standard map results in the so-called
``dominant symmetry line" that contains all the minimax rotational periodic orbits
\cite{meissSymplecticMapsVariational1992}---these are the orbits that are elliptic
for small perturbations. The result is that the sample of initial conditions along the line $\theta_0 = \zeta_0 = 0$ 
includes more regular orbits, the elliptic islands around each of these periodic orbits.
In contrast, the line $\theta_0 = 0.15$ intersects fewer of the elliptic regions around the islands.
Hence, the fraction of chaotic orbits for $\theta_0 = 0$ is less than that for $ \theta_0 = 0.15 $.

Of course, a more general system than \Eq{FareyFields}, e.g., one with added phases in each Fourier term,
would not have this symmetry and would thus be less susceptible to this problem. 
In any case, this issue could be ameliorated by sampling initial conditions on a 2D grid in the Poincar\'e section; of course, this would add the the computational expense. One could keep track of which grid cells are visited by each orbit, so as to reduce the number that need to be considered.


\section{Discussion}

We have investigated the utility of the WBA for distinguishing between regular and chaotic orbits for the two-wave Hamiltonian system, a quasiperiodically forced, dissipative system that has
a strange attractor with no positive Lyapunov exponents, and a model for magnetic field line flow.
It was shown that the WBA is super-convergent when the dynamical system and phase space function are smooth and the dynamics is either conjugate to a rigid rotation with Diophantine rotation vector or more generally satisfies \Eq{cohomologous}. The contrasting, relatively slow convergence of chaotic trajectories provides an efficient discrimination criterion. However, there remain some open questions and interesting further directions.

A first theoretical question is that of the convergence of the weighted Birkhoff average for general ergodic flows.
In each application it was observed that the WBA for chaotic orbits converged relatively slowly in comparison
to the regular orbits. This formed the basis for the WBA as a method to detect chaos.
However, this slow convergence does not yet have a theoretical foundation.
It may be possible to show that \Eq{cohomologous} is not only sufficient for super-polynomial convergence,
but also necessary. If this is true, then it may provide a path forward to theoretically
confirming the slow convergence for chaotic orbits observed in this paper.

One of the benefits of the WBA is that it can provide an accurate computation of the average of a phase space function. 
Indeed, when the average converges, one gets---for free---a good approximation to $\langle h \rangle$.
Consequently, given a physically important $h$, such as rotation number,
its value is computed as a free by-product of the method.
Conversely, if the main goal is to compute an orbit average of some smooth function,
then super-convergence of the WBA on regular orbits, also gives---for free---a criterion
distinguishing between regular and chaotic behavior.

This poses the question: which $h$ is optimal for chaos detection? This appears to be a difficult question.
It is clear that some choices are poor, and this is supported by the convergence theorems for the Birkhoff average
in \cite{Krengel78,Kachurovskii96}. Moreover, an everywhere constant function is also obviously a poor choice
for a different reason: its average over any orbit is the same for any $T$. In some sense, an ideal
function for distinguishing chaos would be constant on regular orbits, but vary on chaotic ones.
In this case the average along the latter should still exhibit the characteristically slow convergence
of the above applications. Of course, if one were able to construct such a function,
then one would already know the orbital structure of the flow, obviating the necessity of a computation. 

This argument suggests that an approximate integral of the system might be a good choice for $h$. 
Such a choice would ensure that $h$ has little variation on regular orbits, still leaving room to 
see the distinction between convergence
for regular and chaotic orbits. We obtained several such approximate integral functions
for the two-wave system in \cite{Duignan20}. However, our preliminary studies using these
approximate integrals as $h$ in two-wave system did not appear to produce a stronger contrast in $\dig_T$  
between chaotic and regular orbits. In the future, we hope to 
investigate the choice of $h$ in the weighted Birkhoff average as a means of detecting chaos.

A further line of future study is that of an effective measure of (non)-integrability. 
This was one of the primary aims of the work in \cite{paulHeatConductionIrregular2022}. 
Such a measure would help in optimizing field configurations by minimize chaos.
There are several improvements of the crude measure of \cref{fig:crudeMeasure} that one could
implement and use to understand chaos in magnetic confinement devices. 

Finally, it was evident in \cref{sec:QPendulum} that the WBA can distinguish between regular and strange 
``non-chaotic attractors''. Thus, the convergence rate distinction for the WBA does not rely on 
exponential divergence of orbits. Future investigation is needed to understand precisely
which types of dynamics this method can accurately discern.

\section*{Acknowledgements}
The authors acknowledge support of the Simons Foundation through grant \#601972 
``Hidden Symmetries and Fusion Energy." 
Useful conversations with E. Sander are gratefully acknowledged. 

\section*{Data Availability}
The data that support the findings of this study are available within the article.


\bibliographystyle{alpha}
\bibliography{WeightedBirkhoff}

\end{document}